\documentclass{article}
\usepackage[dvipsnames]{xcolor}
\usepackage{amsfonts,amsmath,mathabx,fullpage,amssymb,amsthm,tikz,mathtools,lmodern,paralist,mathrsfs,hyperref,stmaryrd,enumitem,tikz-cd,ytableau,tcolorbox,cases}
\usepackage[nottoc]{tocbibind}
\hypersetup{colorlinks, citecolor=red, filecolor=black, linkcolor=blue, urlcolor=blue}
\usetikzlibrary{calc,3d,patterns}

\usepackage{color}
\usepackage[color]{xy}
\usepackage[utf8]{inputenc}
\usepackage{cleveref}

\newtheorem{thmx}{Theorem}

\title{The $cd$-index of semi-Eulerian posets}
\author{Martina Juhnke \thanks{Institut f\"ur Mathematik, Universit\"at Osnabr\"uck, Albrechtstra\ss e 28a, Osnabr\"uck, Germany. \texttt{martina.juhnke@uni-osnabrueck.de}}, Jos\'e Alejandro Samper \thanks{Departamento de Matemáticas, Pontificia Universidad Católica de Chile, Macul, Chile. \texttt{jsamper@uc.cl}}
\thanks{Partially supported by ANID Fondecyt Iniciación grant No. 11221076}, Lorenzo Venturello \thanks{Dipartimento di ingegneria informatica e scienze matematiche, Universit\`{a} di Siena, Italy. \texttt{lorenzo.venturello@unisi.it}}
\thanks{Partially supported by INdAM-GNSAGA and by PRIN 2022S8SSW2}
}
\date{}

\newtheorem{theorem}{Theorem}[section]
\newtheorem{proposition}[theorem]{Proposition}
\newtheorem{lemma}[theorem]{Lemma}
\newtheorem{corollary}[theorem]{Corollary}
\newtheorem{conjecture}[theorem]{Conjecture}

\theoremstyle{definition}
\newtheorem{definition}[theorem]{Definition}
\newtheorem{example}[theorem]{Example}
\newtheorem{remark}[theorem]{Remark}

\begin{document}

\maketitle
\begin{abstract}
We generalize the definition of the $cd$-index of an Eulerian poset to the class of semi-Eulerian posets. For simplicial semi-Eulerian Buchsbaum posets, we show that all coefficients of the $cd$-index are non-negative. This proves a conjecture of Novik for odd dimensional manifolds and extends it to the even dimensional case. 
\end{abstract}

\section{Introduction}

The $cd$-index is a non-commutative polynomial originally defined by Fine to condense the linear relations of the flag $f$-vector of a polytope, namely, the count of the number of flags of faces by dimension. In short, the $cd$-index is a set of coordinates for the smallest linear space containing  all flag $f$-vectors of all polytopes of a fixed dimension. The existence of these coordinates is equivalent to the generalized Dehn-Sommerville (GDS) equations that were discovered by Bayer and Billera in their seminal work on the topic \cite{BaBi}. In fact, those equations hold in greater generality and are a consequence of the fact that face posets of polytopes are Eulerian, i.e., that the Euler characteristic is locally and globally indistinguishable from the one of a sphere of the same dimension. There are numerous results on $cd$-indices in the literature (see for example \cite{Bayer-Survey}).

A bit more formally, the $cd$-index is a homogeneous polynomial in two noncommutative variables $c$ and $d$ of degrees $1$ and $2$ respectively. In the first paper that popularized the $cd$-index, Bayer and Klapper \cite{BaKla} conjectured that the $cd$-index of Gorenstein* posets has nonnegative coefficients. These are Eulerian posets whose order complex is Cohen-Macaulay and include face lattices of polytopes and of regular CW-decompositions of homology spheres. Quickly after the conjecture appeared Stanley \cite{Stanley94} proved its validity for spherically shellable posets, a class that includes all polytopes and for Gorenstein* simplicial posets. The full resolution of the conjecture came from the work of Karu \cite{Karu} using sheaves on posets to fully exploit the homological restrictions imposed on Gorenstein* posets. What is notable about this condition is that Bayer \cite{Bayer-Signs} showed later that for general Eulerian posets there are certain coefficients of the $cd$-index that can be arbitrarily negative, which means that the homological restrictions play an importante role in the positivity phenomenon.

 The main goal of this article is to define a suitable $cd$-index for semi-Eulerian posets. These are bounded graded posets such that the Euler characteristic of any interval not containing the maximal element coincides with the Euler characteristic of a sphere of the right dimension. In other words, semi-Eulerian posets are locally Eulerian posets and carry an additional number, their Euler characteristic, which may be different from that of spheres. Examples of this are face lattices of triangulations and of regular CW decompositions of (homology) manifolds. The $f$-vector theory for this class of posets exhibits many similarities to that of Eulerian posets, but the $cd$-index does not directly generalize. Nonetheless, the affine span of the collection of flag $f$-vectors of semi-Eulerian posets with a fixed dimension and Euler characteristic is parallel translate of the affine span of flag $f$-vectors of Euclidean posets, as observed independently by Ehrenborg \cite{Ehrenborg-KEUL} and Swartz \cite{Sw:SpheresToManifolds}. In fact, this affine span is described by a slight modification of the generelized Dehn-Sommerville equations of Bayer and Billera that takes into accound the additional data of the Euler characteristic. It is therefore reasonable to expect that a variant of the $cd$-index that extends the positivity theory exists. Nonetheless some care is needed: for example Ehrenborg \cite{Ehrenborg-KEUL} provided a certain definition that is natural from the point of view of (co)algebras, but has coefficients that are negative and sometimes rational instead of integral.

We propose a different approach following Swartz: namely a small modification of one of the entries of the flag $f$-vector of a semi-Eulerian poset $P$ (or equivalently of one single monomial of the usual flag $f$-polynomial $\chi_P(a,b)$) translates the flag $f$-vector of a semi-Eulerian poset to span of Eulerian posets admitting a $cd$-index. This translation is clearly not unique and a good justification of our choice is neeed. In addition to the appealing simplicity that it has at the level of $f$-vectors, our choice of translation leads to to an  identity for simplicial posets that generalizes the main trick of Stanley in his orignal proof that spherically shellable posets have a non-negative $cd$-index.
\begin{thmx}[\Cref{thm:cd}, \Cref{thm:CDvsH}]
    Let $P$ be a semi-Eulerian poset of rank $n+1$. The polynomial 
    \[\chi_P'(a,b):=\chi_P(a,b)+(\chi(\mathbb{S}^{n-1})-\chi(P))a\cdots ab\]
    admits a $cd$-index $\Phi_P(c,d)$. Moreover, if $P$ is simplicial we have that 
    \[
        \Phi_P(c,d) = \sum_{i=0}^{n-1} h_i(P)\check\Phi^n_i(c,d),
     \]
     with $\check\Phi^n_i$ as defined in \Cref{sect:cd def} and $(h_0(P), h_1(P), \dots h_n(P))$ is the usual $h$-vector.
\end{thmx}

The identiy above will be a key step in for the next main theorem. We remark that it is actually non-trivial: if one wants the identity to hold, then the translation of span of the flag $f$-vectors of simplicial semi-Eulerian posets of a fixed Euler characteristic is not the span of $f$-vectors of Eulerian simplicial posets. Thus using the result of Stanley requires lengthy computations. The consequences of this identity, together with a careful analysis of a family of polynomials which we identify in \Cref{sec : P poly}, turn out to be quite interesting and lead to the main result of the article. 

\begin{thmx}[\Cref{thm:BoundPhi}]\label{thm: B}
    The coefficients of the $cd$-index of a semi-Eulerian Buchsbaum simplicial poset are nonnegative.
\end{thmx}

The theorem above applies for example when the poset has a geometric realization that is a homology manifold of a field. If the poset is homeomorphic to an odd dimensional manifold, the theorem answers a question of Novik \cite{Novik:cd} about . Even more, \Cref{thm:BoundPhi} gives stronger lower bounds for the coefficients of the $cd$-index which also take into account the topology of the order complex of $P$.

Moreover, it has interesting consequences on the face numbers of the order complex $\Delta(\overline{P})$ of an Eulerian Buchsbaum simplicial poset $P$ of odd rank. In fact, the \emph{$\gamma$-polynomial} of $\Delta(\overline{P})$, whose coefficients $\gamma_i(\Delta(\overline{P}))$ are particular linear combinations of its face numbers, can be computed simply as an evaluation of $\Phi_P(c,d)$ (see \cite[Section 2.3]{Gal}). As a consequence of \Cref{thm: B} we obtain the following.
\begin{thmx}[\Cref{cor : manifolds}]\label{thm: C}
    Let $P$ be an Eulerian Buchsbaum simplicial poset of rank $2k+1$. Then 
    \[\gamma_i(\Delta(\overline{P}))\geq 0,\]
    for every $i=0,\dots,k$. In particular,
    \[
     (-1)^k\sum_{i=0}^{2k} (-1)^{i}h_i(\Delta(\overline{P}))\geq 0.
    \]
\end{thmx}
The nonnegativity of the coefficients of the $\gamma$-polynomial has been conjectured for flag homology spheres. \Cref{thm: C} shows that this phenomenon holds, for instance, for barycentric subdivisions of odd-dimensional manifolds as well. The second statement in \Cref{thm: C} is meant to underline the most interesting of the linear inequalities which correspond to the nonnegativity of the numbers $\gamma_i(\Delta(\overline{P}))$, which goes back to a slightly older conjecture of Charney and Davis \cite{CD}.
\section{Preliminaries}
The aim of this section is to introduce required notion and provide background on the (classical) $cd$-index. 

We start with some poset terminology. In the following, without further mention, we always assume that a finite graded poset has a unique minimal and maximal element $\hat{0}$ and $\hat{1}$, respectively. 
A finite graded poset $P$ is called \emph{semi-Eulerian} if every proper interval, $[s,t]\neq [\hat{0},\hat{1}]=P$ with more than one element has the same number of elements of odd rank as of even rank. A semi-Eulerian poset is called \emph{Eulerian} if, in addition, $P$ has the same number of elements of odd rank as of even rank. The condition that the number of elements of odd rank equals the number of elements of even rank in an interval $[s,t]$ can be reinterpreted as follows: If $\mathrm{rk}$ denotes the rank function of $P$ and $\mu_P$ its M\"obius function, then $\mu_P(s,t)=(-1)^{\mathrm{rk}(t)-\mathrm{rk}(s)}$. In particular, this means that, locally, semi-Eulerian posets have the Euler characteristic of a sphere of the corresponding dimension. In the Eulerian case this is also true globally.  Examples of Eulerian posets are provided by face posets of polytopes, of regular CW spheres (in particular, simplicial spheres) and of odd-dimensional simplicial manifolds after adjoining a maximal element $\hat{1}$ (or more generally, of Buchsbaum complexes). Face posets of regular CW complexes homeomorphic to even-dimensional manifolds are not Eulerian in general but semi-Eulerian. 

Let $P$ be a poset of rank $n+1$. For a set $S=\{\ell_1<\ell_2<\cdots <\ell_k\}\subseteq [n]$, we set 
\[
f_P(S)\coloneqq \#\{\hat{0}<x_1<x_2<\cdots <x_k<\hat{1}~:~x_i\in P,\; \mathrm{rk}(x_i)=\ell_i \mbox{ for } 1\leq i\leq k\}, 
\]
i.e., $f_P(S)$ counts the number of chains in $P$ whose elements have the ranks specified by $S$. The vector $(f_P(S))_{S\subseteq [n]}$ is called the \emph{flag $f$-vector} of $P$. The alternating sum of the numbers of length $1$ chains is called the \emph{(reduced) Euler characteristic} $\chi(P)$ of $P$. More precisely,
\[
\chi(P)=\sum_{j=1}^n(-1)^{j-1}f_{\{j\}}(P).
\]
Flag $f$-vectors of Eulerian posets satisfy the so-called \emph{generalized Dehn-Sommerville relations}:

\begin{theorem}\cite[Theorem 2.1]{BaBi}\label{thm:gen DS}
  Let $P$ be an Eulerian poset of rank $n+1$ and $S\subseteq [n]$. If $\{i,k\}\subseteq S\cup\{0,n+1\}$ for $i<k-1$ and $S$ contains no $i<j<k$, then
  \[
  \sum_{j=i+1}^{k-1}(-1)^{j-i-1}f_{S\cup \{j\}}(P)=f_S(P)(1-(-1)^{k-i-1}).
  \]
\end{theorem}

Moreover, these relations can be used to show that the affine space spanned by the flag $f$-vectors of rank $n+1$ Eulerian posets has dimension $F_{n}-1$, where $(F_n)_{n\geq 0}$ is the \emph{Fibonacci sequence} with $F_0=F_1=1$. Remarkably, flag $f$-vectors of convex polytopes, i.e., of their face posets, suffice to span this space \cite[Theorem 2.6]{BaBi}.

Instead of the flag $f$-vector, one often considers the following invertible linear transformation. If $P$ is a rank $n+1$ poset with flag $f$-vector $(f_S(P))_{S\subseteq [n]}$, the \emph{flag $h$-vector} $(h_S(P))_{S\subseteq [n]}$ of $P$ is defined via
\begin{equation}\label{eq : flag f to flag h}
    h_S(P)=\sum_{T\subseteq S}(-1)^{|S\setminus T|}f_T(P).
\end{equation}

One can encode the flag $h$-vector into a single polynomial in non-commuting variables $a$ and $b$ as follows: For a subset $S\subseteq [n]$, we write $u_S$ for the monomial $u_1u_2\cdots u_n$ with $u_i=a$ if $i\notin S$ and $u_i=b$ if $i\in S$. The polynomial
\[
\Psi_P(a,b)=\sum_{S\subseteq [n]}h_S(P)u_S
\]
is called the \emph{$ab$-polynomial} of $P$. Setting $\chi_P(a,b)=\sum_{S\subseteq [n]}f_S(P)u_S$, it is not hard to see that
\[
\Psi_P(a,b)=\chi_P(a-b,b).
\]
In the following, we will refer to $\chi_P(a,b)$ also as the \emph{chain polynomial} of $P$. 
We now recall the definition of the $cd$-index for Eulerian posets. 
\begin{definition}
Let $P$ be a graded Eulerian poset of rank $n+1$. The \emph{$cd$-index} of $P$ is the polynomial $\Phi_P(c,d)$ in non-commuting variables $c$ and $d$ such that
\[
\Phi_P(a+b,ab+ba)=\Psi_P(a,b).
\]
\end{definition}
Indeed, Fine (see \cite{BaKla,Stanley94}) showed that such a polynomial exists (and is unique) for Eulerian posets. Setting $\deg(c)=1$ and $\deg(d)=2$, the $cd$-index becomes a polynomial of total degree $n$ in $c$ and $d$. So as to be able to generalize the $cd$-index to semi-Eulerian posets, we need to introduce some further notation.

Given a polynomial $p(a,b)$ in non-commuting variables $a$ and $b$, we say that \emph{$p$ admits a $cd$-index} if $p(a-b,b)$ can be be written as a (unique) polynomial $q(c,d)$ with $c=a+b$ and $d=ab+ba$, i.e., $p(a-b,b)=q(a+b,ab+ba)$. Similarly, we say that \emph{$p$ satisfies the generalized Dehn-Sommerville relations} if its sequence of coefficients satisfies the relations in \Cref{thm:gen DS}. The next proposition, which will be crucial, is a reformulation of \cite[Theorem 4]{BaKla} (see also \cite{Fine}).

\begin{proposition}\label{prop:existenceCD}
 A polynomial $p$ in non-commuting variables $a$ and $b$ admits a $cd$-index if and only if it satisfies the generalized Dehn-Sommerville relations.
\end{proposition}
To simplify notation (in particular in \Cref{sect:cd def}) we will also say that a vector $(b_S)_{\S\subseteq [n]}$ admits a $cd$-index, if the associated polynomial $\sum_{S\subseteq [n]}b_Su_S$ does.\\

Sometimes we will speak about (semi)-Eulerian simplicial complexes, rather than (semi)-Eulerian posets. A simplicial complex $\Delta$ is called \emph{semi-Eulerian} or \emph{Eulerian} if its face poset with a unique maximal element $\hat{1}$ adjoined, $P(\Delta)$, has this property. The flag $f$-vector, flag $h$-vector and \emph{$cd$-index} of $\Delta$ (if it exists) is then defined to be the corresponding invariant for $P(\Delta)$. More precisely, we set
\[
f_S(\Delta)=f_S(P(\Delta)), \quad h_S(\Delta)=h_S(P(\Delta))\quad \mbox{and}\quad \Phi_\Delta(c,d)=\Phi_{P(\Delta)}(c,d),
\]
 where $\dim\Delta=n-1$. 

Given a simplicial complex $\Delta$ and a face $F\in \Delta$, the \emph{link of $F$ in $\Delta$} is the subcomplex $\mathrm{lk}_\Delta(F)=\{G\in \Delta~:~G\cup F\in \Delta; G\cap F=\emptyset\}$. A pure simplicial complex $\Delta$ is called \emph{Buchsbaum} (over a field $\mathbb{F}$) if $\widetilde{H}_i(\mathrm{lk}_\Delta(F);\mathbb{F})=0$ if $i\neq\dim \Delta-|F|$ (cf., \cite[Theorem 3.2]{NS-Socles}), where $\widetilde{H}_i(\mathrm{lk}_\Delta(F);\mathbb{F})$ denotes the reduced simplicial homology over $\mathbb{F}$. 
A simplicial poset $P$ with a unique minimal and maximal element $\hat{0}$ and $\hat{1}$, respectively, is called \emph{Buchsbaum} over $\mathbb{F}$ if the order complex $\Delta(\overline{P})$ of $\overline{P}=P\setminus\{\hat{0},\hat{1}\}$ is a Buchsbaum simplicial complex over $\mathbb{F}$.

A simplicial complex $\Delta$ is a \emph{($\mathbb{F}$-)homology manifold} if 
\[
\widetilde{H}_i(\mathrm{lk}_\Delta(F);\mathbb{F})=\begin{cases}
    \mathbb{F},\quad \mbox{if } i=\dim\Delta-|F|=\dim\mathrm{lk}_{\Delta}(F),\\
    0,\quad \mbox{otherwise}.
\end{cases}
\] 
Homology manifolds include simplicial manifolds, i.e., triangulations of topological manifolds. Clearly, every $\mathbb{F}$-homology manifold is a Buchsbaum complex over $\mathbb{F}$.

Given an $(n-1)$-dimensional simplicial complex, we let $f_i(\Delta)$ denote the number of $i$-dimensional faces of $\Delta$. The vector $f(\Delta)=(f_{-1}(\Delta),f_0(\Delta),\ldots,f_{n-1}(\Delta))$ is called the \emph{$f$-vector} of $\Delta$. The \emph{$h$-vector} $h(\Delta)=(h_0(\Delta),h_1(\Delta),\ldots,h_n(\Delta))$, where 
\[
h_j(\Delta)=\sum_{i=0}^j(-1)^{j-i}\binom{n-i}{j-i}f_{i-1}(\Delta),
\]
is called the $h$-vector of $\Delta$. The \emph{Euler characteristic} $\chi(\Delta)$ of $\Delta$ is defined as $\chi(\Delta)=\sum_{i=0}^{n-1}(-1)^if_i(\Delta)$. We note that we have $\chi(\Delta)=\chi(P(\Delta))$, by definition. The \emph{$i$-th (reduced) Betti number} $\widetilde{\beta}_i(\Delta;\mathbb{F})$ (over $\mathbb{F}$) is defined as $\widetilde{\beta}_i(\Delta;\mathbb{F})=\dim_\mathbb{F}\widetilde{H}_i(\Delta;\mathbb{F})$. 

The following lower bound for the $h$-numbers of Buchsbaum simplicial posets will be crucial in \Cref{sec: nonnegative}

\begin{theorem}\cite[Theorem 6.4]{NS-Socles}\label{thm:NovSwa}
  Let $P$ be a  Buchsbaum simplicial poset (with $\hat{0}$ and $\hat{1}$) of rank $n+1$ (over $\mathbb{F}$). Then 
  \[
  h_i(P)\geq \binom{n}{i}\sum_{j=1}^i (-1)^{i-j}\widetilde{\beta}_{j-1}(\Delta(\overline{P});\mathbb{F}).
  \]
\end{theorem}

\section{The $cd$-index for a semi-Eulerian poset}\label{sect:cd def}
The goal of this section is to introduce a $cd$-index for semi-Eulerian posets. 
Inspired by \Cref{prop:existenceCD}, the main idea for achieving this is to modify the flag $f$-vector (and consequently, the flag $h$-vector) in such a way that it satisfies the generalized Dehn-Sommerville relations. The next proposition, which is a direct consequence of the proof of the generalized Dehn-Sommerville relations in \cite{BaBi}, is crucial:

\begin{proposition}\label{prop:DS semiEu}
  Let $P$ be a semi-Eulerian poset of rank $n+1$ and $S\subseteq [n]$. If $S\neq \emptyset$, $\{i,k\}\subseteq S\cup\{0,n+1\}$ for $i<k-1$ and $S$ contains no $i<j<k$, then
  \[
  \sum_{j=i+1}^{k-1}(-1)^{j-i-1}f_{S\cup \{j\}}(P)=f_S(P)(1-(-1)^{k-i-1}).
  \]
\end{proposition}
 Equivalently, the only generalized Dehn-Sommerville relation that is not satisfied by the flag $f$-vector of a semi-Eulerian poset $P$ of rank $n+1$ is the unique one corresponding to $S=\emptyset$ (and consequently, $i=0$, $k=n+1$). Namely,
 \begin{equation}\label{eq:DS not valid}
 \sum_{j=1}^n(-1)^{j-1}f_{\{j\}}(P)=f_{\emptyset}(P)(1-(-1)^n).
 \end{equation}
 We note that the left-hand side of \eqref{eq:DS not valid} equals the Euler characteristic of $P$ and that the right-hand side simplifies to $1-(-1)^n$ as $f_{\emptyset}(P)=1$. In particular, since $1-(-1)^n$ equals the Euler characteristic $\chi(\mathbb{S}^{n-1})$ of an $(n-1)$-dimensional sphere, one sees that \eqref{eq:DS not valid} is the usual Euler relation for spheres. 

 We can now define a slightly modified flag $f$-vector for graded posets.
 \begin{definition}\label{def: f'}
     Let $P$ be a graded poset of rank $n+1$ with a unique minimal and maximal element $\hat{0}$ and $\hat{1}$. For $S\subseteq [n]$, let
     \[
     f'_S(P)=\begin{cases}
     f_{\{n\}}(P)+(\chi(\mathbb{S}^{n-1})-\chi(P)),\quad\mbox{if } S=\{n\}\\
     f_S(P), \quad\mbox{otherwise.}
     \end{cases}
     \]
     The polynomial $\chi'_P(a,b)=\sum_{S\subseteq [n]}f'_S(P)u_S$ is called \emph{modified chain polynomial}.
 \end{definition}
 We note that $\chi_P(a,b)=\chi'_P(a,b)$ if $P$ is an Eulerian poset. The next result is almost immediate:
 \begin{theorem}\label{thm:cd}
     Let $P$ be a semi-Eulerian poset of rank $n+1$. The modified chain polynomial $\chi'_P(a,b)$ admits a $cd$-index, i.e., there exists a polynomial $\Phi_P(c,d)$ in non-commuting variables $c$ and $d$ such that
\[
\Phi_P(a+b,ab+ba)=\chi'_P(a-b,b).
\]
 \end{theorem}
 We will refer to the polynomial $\Phi_P(c,d)$ as the \emph{$cd$-index} of $P$. This is justified by the fact that it coincides with the usual $cd$-index in case of Eulerian posets.

 \begin{proof}
 If $n$ is even, $P$ is Eulerian and the statement follows just from the existence of the usual $cd$-index in this case. 
 
 Now let $n$ be odd. By \Cref{prop:existenceCD} the existence of the $cd$-index is equivalent to the polynomial $\chi'_P(a,b)$ satisfying the generalized Dehn-Sommerville relations. All such not involving $f'_{\{n\}}$ are valid for the usual flag $f$-vector by \Cref{prop:DS semiEu} and hence remain valid for the modified flag $f$-vector $(f'_{S}(P))_{S\subseteq [n]}$, by definition. There are only two generalized Dehn-Sommerville relations involving $f'_{\{n\}}(P)$. We first consider the  one corresponding to Euler's formula for spheres (see \eqref{eq:DS not valid}). We have
 \begin{align*}
 \sum_{j=1}^{n}(-1)^{j-1}f'_{\{j\}}(P)=&\sum_{j=1}^{n}(-1)^{j-1}f_{\{j\}}(P)+(-1)^{n-1}(\chi(\mathbb{S}^{n-1})-\chi(P))\\
 =&\chi(P)+\chi(S^{n-1})-\chi(P))=\chi(\mathbb{S}^{n-1})=1-(-1)^n,
 \end{align*}
 as desired. 
 
 The only other generalized Dehn-Sommerville relation that is different for $f$ and $f'$ is the unique one with $S=\{n\}$ (and consequently, $i=0$ and $k=n$). For $f'$ this relation states the following:
 \[
 \sum_{j=1}^{n-1}(-1)^{j-1}f'_{\{n,j\}}=f'_{\{n\}}(1-(-1)^{n-1})=0,
 \]
 where the last equality follows from the fact that $n$ is odd. Therefore, the above relation does not depend on $f'_{\{n\}}$ and is equivalent to the corresponding relation for $(f_S(P))_{S\subseteq [n]}$. Since the latter is valid by \Cref{prop:DS semiEu}, the claim follows.
 \end{proof}
 We want to remark that \Cref{thm:cd} is basically a restatement of Theorem 3.12 in \cite{Sw:SpheresToManifolds}. However, we decided to include a proof in order to make this article self-contained.

Together with \cite[Theorem 3.12]{Sw:SpheresToManifolds} the previous theorem implies that the following geometric relation between the affine spans of flag $f$-vectors of semi-Eulerian posets with different Euler characteristic.

\begin{lemma}
    Let $k,\ell$ be integers. Let $A_{SE}(k,n)$ and $A_{SE}(\ell,n)$ denote the affine span of flag $f$-vectors of semi-Eulerian posets of rank $n+1$ and Euler characteristic $k$ and $\ell$, respectively. Then $A_{SE}(k,n)$ and $A_{SE}(\ell,n)$ are parallel translates of each other. More precisely,
    \[
    A_{SE}(k,n)=A_{SE}(\ell,n)+(k-\ell)\mathbf{e}_{\{n\}},
    \]
    where $\mathbf{e}_{\{n\}}$ denotes the unit vector that has a $1$ at the position corresponding to $\{n\}$.
\end{lemma}

In view of the previous lemma, it is clear that we could have taken other corrections to enforce the generalized Dehn-Sommerville equations. However, it will turn out that our choice is not only extremely simple on the level of flag $f$-vectors but also comes with other nice properties. One of the main ingredients to show non-negativity of the $cd$-index for odd-dimensional manifolds in \cite{Novik:cd} was a certain recursion by Stanley \cite[Theorem 3.1]{Stanley94}, expressing the $cd$-index in terms of the $h$-vector and some specific polynomials $\check\Phi_i^n(c,d)$. We now recall the definition of these polynomials and then show that the mentioned recursion remains valid in our more general setting.

In the following, we restrict our attention to \emph{simplicial} posets. Recall that a poset $P$ is \emph{simplicial}, if every interval $[\hat{0},x]$ with $x\in P\setminus\{ \hat{1}\}$ is isomorphic to  a Boolean algebra. If $P$ is the augmented face poset of a regular CW complex, simpliciality just means that we are in the special case of simplicial complexes. Given a simplicial poset $P$ of rank $n+1$, we let $f_{i-1}(P)$ count the number of elements $x\in P$ such that $[\hat{0},x]$ is isomorphic to the Boolean algebra of rank $i$. The \emph{$f$-vector} $f(P)=(f_{-1}(P),f_0(P),\ldots,f_{n-1}(P))$ of $P$ is called the \emph{$f$-vector} of $P$. Note that when $P$ is simplicial its flag $f$-vector is completely determined by its $f$-vector. Indeed, we have that
\begin{align}
    f_S(P)=\binom{s_k}{s_1,s_2-s_1,\dots,s_k-s_{k-1}}f_{s_k-1}(P),
\end{align}
for $S = \{s_1<\cdots<s_k\}$, and $f_{\emptyset}(P)=f_{-1}(P)$. In general, the flag $f$-vector contains more information as $f_{i}(P)=f_{\{i+1\}}(P)$ for $0\leq i\leq n-1$. Moreover, if $P$ is the augmented face poset of a simplicial complex $\Delta$, then $f(P)=f(\Delta)$. The \emph{$h$-vector} $h(P)=(h_0(P),h_1(P),\ldots,h_n(P))$ of a simplicial poset $P$ of rank $n+1$ is defined via  
\[
\sum_{i=0}^{n} f_{i-1}(P)(x-1)^{n-i}=\sum_{i=0}^n h_i(P)x^{n-i}.
\]
If $P$ is the augmented face poset of a simplicial complex $\Delta$, then $h(P)=h(\Delta)$. However, one should not confuse the $h$-vector of $P$ with the $h$-vector of its order complex.

\begin{remark}\label{rem: important}
We have already mentioned that simplicial Eulerian posets admit a $cd$-index. Equivalently, their flag $f$-vectors satisfy the generalized Dehn-Sommerville relations, see \Cref{thm:gen DS}. Moreover, the only linear relations satisfied by the $h$-vector of any simplicial Eulerian poset of rank $n+1$ are the classical Dehn-Sommerville relations $h_i=h_{n-i}$ for $0\leq i\leq n$. Since the $h$-vector of a is mapped by a linear map to the flag $f$-vector, it follows that under this map any vector satisfying the classical Dehn-Sommerville relations (i.e., lying in the span of $h$-vectors of simplicial Eulerian posets) is mapped to a vector satisfying the generalized Dehn-Sommerville relations. As such the latter admits a $cd$-index. The proof of \Cref{thm:CDvsH} will significantly rely on this fact.
\end{remark}

For a CW complex $\Gamma$, whose geometric realization is an $n$-dimensional ball, the \emph{semisuspension} of $\Gamma$ is the CW complex obtained from $\Gamma$ by attaching a new $n$-cell $\tau$ along the boundary of $\Gamma$, i.e., $\partial\tau=\partial\Gamma$. 

Let $\Lambda^n$ be the boundary complex of an $n$-simplex with facets $\sigma_0,\sigma_1,\ldots,\sigma_n$. For $0\leq i\leq n-1$, let $\Lambda_i^n$ be the $(n-1)$-dimensional CW complex that is the semisuspension of the simplicial complex $\Gamma^n_i$ generated by $\sigma_0,\sigma_1,\ldots,\sigma_i$. In particular, $\Lambda_{n-1}^{n}=\Lambda^n$. One now sets $\check\Phi_i^n=\Phi_{\Lambda_i^n}-\Phi_{\Lambda_{i-1}^n}$ with $\check\Phi_0^n=\Phi_{\Lambda_0^n}$ and $\check\Phi_n^n=0$. With these definition we can now state the desired recursion (cf. \cite[Theorem 3.1]{Stanley94} for the corresponding statement for Eulerian simplicial posets).
\begin{theorem}\label{thm:CDvsH}
    Let $P$ be a semi-Eulerian simplicial poset of rank $n+1$. Then 
    \begin{equation}\label{eqn:PSIvsH}
        \Phi_P(c,d) = \sum_{i=0}^{n-1} h_i(P)\check\Phi^n_i(c,d).
    \end{equation}
\end{theorem}
In order to prove \Cref{thm:CDvsH} we apply to $\check\Phi^n_i=\check\Phi^n_i(c,d)$ the linear map which sends the $cd$-index of an Eulerian poset to its flag $f$-vector. We call the vector obtained in this way the \emph{flag $f$-vector associated to $\Phi^n_i(c,d)$} and use $f_S(\check\Phi_i^n)$ to denote its entries for $\emptyset\subsetneq S\subsetneq [n]$.

\begin{lemma}\label{lem: identity hard}
   Let $n$ be odd and let $(f_S(\check\Phi_i^n))_{S\subseteq[n]}$ be the flag $f$-vector associated with the $cd$-polynomial $\check\Phi_i^n(c,d)$. Then for $S=\{s_1<\cdots<s_k\}\subseteq [n]$ the following identity holds
    \begin{equation}\label{eq: identity hard}
        \sum_{i=\lfloor\frac{n}{2}\rfloor+1}^{n-1}(-1)^i\binom{n}{i}f_S(\check\Phi_i^n) = \binom{n}{s_k}\binom{s_k}{s_1,s_2-s_1,\dots,s_k-s_{k-1}}\sum_{\lfloor\frac{n}{2}\rfloor+1\leq j \leq s_k}(-1)^{j}\binom{s_k}{j}+\delta_{S,\{n\}},
    \end{equation}
    where $\delta_{S,\{n\}}=1$ if $S=\{n\}$ and $\delta_{S,\{n\}}=0$, otherwise.
\end{lemma}
\begin{proof}
    Since the map which takes a $cd$-polynomial to the associated flag $f$-vector is linear we have that 
    \[
        f_S(\check\Phi_i^n)=f_S(P(\Lambda^n_i))-f_S(P(\Lambda^n_{i-1})).
    \]
    The key observation is that the augmented face posets $P(\Lambda^n_i)$ are not far from being simplicial. More precisely, all proper lower order ideals are isomorphic to Boolean posets, except the one corresponding to the special $(n-1)$-dimensional cell added in the semisuspension. In particular, the number of chains in such interval whose elements have rank in a fixed set $S=\{s_1<\cdots<s_k\}\subseteq [n]$ can be expressed in terms of the $f$-vector of $\partial\Lambda^n_i$ as follows:  
    \begin{equation}\label{eq: identity 1}
        f_S(P(\Lambda^n_i))=\begin{cases}
            1 & \mbox{ if } S=\emptyset\\
            \binom{s_k}{s_1,s_2-s_1,\dots,s_k-s_{k-1}}f_{s_k-1}(\Lambda^n_i) & \mbox{ if }s_k<n\\
            \binom{n}{s_1,s_2-s_1,\dots,n-s_{k-1}}(f_{n-1}(\Lambda^n_i)-1)+\binom{s_{k-1}}{s_1,s_2-s_1,\dots,s_{k-1}-s_{k-2}} f_{s_{k-1}-1}(\partial \Gamma^{n}_i) & \mbox{ if } s_k=n, S\neq\{n\}\\
            f_{n-1}(\Lambda^n_i) & \mbox{ if } S=\{n\}.
        \end{cases}
    \end{equation}
    Moreover, by counting the number of new faces introduced with a shelling we obtain
    \begin{equation}\label{eq: identity 2}
        f_{s_k-1}(\Lambda^n_i)-f_{s_k-1}(\Lambda^n_{i-1})=\begin{cases}
            0 & s_k<i \\
            \binom{n-i}{s_k-i} & s_k\geq i.
        \end{cases}
    \end{equation}
    Hence $f_S(\check\Phi_i^n)=0$ if $s_k<i$. We now distinguish between three cases.
    
    {\sf Case 1:} Assume that $s_k<n$. Combining \eqref{eq: identity 1} and \eqref{eq: identity 2} we obtain
    \begin{equation*}
        \sum_{i=\lfloor\frac{n}{2}\rfloor+1}^{n-1}(-1)^i\binom{n}{i}f_S(\check\Phi_i^n) = \sum_{i=\lfloor\frac{n}{2}\rfloor+1}^{s_k} (-1)^i\binom{n}{i} \binom{s_k}{s_1,s_2-s_1,\dots,s_k-s_{k-1}}\binom{n-i}{s_k-i},
    \end{equation*}
    which is easily seen to coincide with the right-hand side of \eqref{eq: identity hard}.
    
    {\sf Case 2:} Assume that $s_k=n$ and $S\neq\{n\}$. By tracing how the boundaries of $\Gamma^n_i$ and $\Gamma^n_{i-1}$ are related we obtain
    \[
        f_{s_{k-1}-1}(\partial\Gamma^n_i)-f_{s_{k-1}-1}(\partial\Gamma^n_{i-1})=\binom{n-i}{n-s_{k-1}}-\binom{i}{n-s_{k-1}}.
    \]
    Substituting this in the left-hand side of \eqref{eq: identity hard} and using \eqref{eq: identity 1} we conclude
    \begin{align*}
    \sum_{i=\lfloor\frac{n}{2}\rfloor+1}^{n-1}(-1)^i\binom{n}{i}f_S(\check\Phi_i^n)
    &=\sum_{i=\lfloor\frac{n}{2}\rfloor+1}^{n-1}(-1)^i\binom{n}{i}\binom{n}{s_1,s_2-s_1,\dots,n-s_{k-1}}\\
    &+ 
    \binom{s_{k-1}}{s_1,s_2-s_1,\dots,s_{k-1}-s_{k-2}}\sum_{i=\lfloor\frac{n}{2}\rfloor+1}^{n-1}(-1)^{i}\binom{n}{i}\left(\binom{n-i}{n-s_{k-1}}-\binom{i}{n-s_{k-1}}\right)\\
    &=\sum_{i=\lfloor\frac{n}{2}\rfloor+1}^{n-1}(-1)^i\binom{n}{i}\binom{n}{s_1,s_2-s_1,\dots,n-s_{k-1}}\\
    &+
    \binom{n}{s_{k-1}}\binom{s_{k-1}}{s_1,s_2-s_1,\dots,s_{k-1}-s_{k-2}}\sum_{i=\lfloor\frac{n}{2}\rfloor+1}^{n-1}(-1)^{i}\left(\binom{s_{k-1}}{i}-\binom{s_{k-1}}{n-i}\right)\\
    &=\sum_{i=\lfloor\frac{n}{2}\rfloor+1}^{n-1}(-1)^i\binom{n}{i}\binom{n}{s_1,s_2-s_1,\dots,n-s_{k-1}}-\binom{n}{s_1,s_2-s_1,\dots,n-s_{k-1}}\\
    &=\sum_{i=\lfloor\frac{n}{2}\rfloor+1}^{n}(-1)^i\binom{n}{i}\binom{n}{s_1,s_2-s_1,\dots,n-s_{k-1}},
    \end{align*}
    where for the third equality we have used that $\sum_{i=\lfloor\frac{n}{2}\rfloor+1}^{n-1}(-1)^{i}\left(\binom{s_{k-1}}{i}-\binom{s_{k-1}}{n-i}\right)=\sum_{j=1}^{n-1}(-1)^j\binom{s_{k-1}}{j}=-1+\sum_{j=0}^{n-1}(-1)^j\binom{s_{k-1}}{j}=-1$. 
    As the last expression above equals the right-hand side of \eqref{eq: identity hard}, the claim follows in this case.
    
    {\sf Case 3:} Assume that $S=\{n\}$. In this case the left-hand side of \eqref{eq: identity hard} equals
    \begin{align*}
    \sum_{i=\lfloor\frac{n}{2}\rfloor+1}^{n-1}(-1)^i\binom{n}{i}f_S(\check\Phi_i^n)&=\sum_{i=\lfloor\frac{n}{2}\rfloor+1}^{n-1}(-1)^i\binom{n}{i}=\sum_{i=\lfloor\frac{n}{2}\rfloor+1}^{n}(-1)^i\binom{n}{i}+1,
    \end{align*}
    which also coincides with the right-hand side of \eqref{eq: identity hard}. This concludes the proof.
\end{proof}
Using \Cref{lem: identity hard} we are now ready to prove \Cref{thm:CDvsH}.

\begin{proof}[Proof of \Cref{thm:CDvsH}]
    If $n$ is even, then $P$ is Eulerian, and the statement is \cite[Theorem 3.1]{Stanley94}. Assume now that $n$ is odd, and consider the vector $(f'_S(P))_{S\subseteq[n]}$ as defined in \Cref{def: f'}. By definition, we have that $\Phi_P(a+b,ab+ba)=\Phi_P(c,d)=\chi'_P(a-b,b)$. However, $(f'_S(P))_{S\subseteq[n]}$ does not lie in the affine span of flag $f$-vectors of Eulerian \emph{simplicial} posets in general. To circumvent this issue we let $\mathbf{w}\coloneqq (\chi(P)-\chi(\mathbb{S}^{n-1}))\sum_{i=\lfloor\frac{n}{2}\rfloor+1}^{n} (-1)^{i}\binom{n}{i}\mathbf{e}_i$, where $\mathbf{e}_i$ denotes the $i$-th standard unit vector in $\mathbb{R}^n$. We consider the following vector in $\mathbf{v}\in \mathbb{R}^{n}$
     \begin{align*}
        \mathbf{v}&\coloneqq h(P)-\mathbf{w}.
    \end{align*}
    By a result of Klee \cite{Klee_1964} (see also \cite[Theorem 3.1]{Sw:SpheresToManifolds}) $\mathbf{v}$ satisfies the (classical) Dehn-Sommerville relations, and it hence lies in the span of $h$-vectors of Eulerian simplicial posets (cf., \Cref{rem: important}). 
    Applying to $\mathbf{v}$ the linear map which sends the $h$-vector of a simplicial poset to its flag $f$-vector, we obtain a vector $\mathbf{g}\coloneqq(\mathbf{g}_S)_{S\subseteq [n]}$ given by
    \[
       \mathbf{g}_S = \binom{s_k}{s_1,s_2-s_1,\dots,s_k-s_{k-1}}\sum_{j\leq s_k}\binom{n-j}{n-s_k}\mathbf{v}_j  
    \]
    for $S=\{s_1<\cdots <s_k\}$. By the previous reasoning (see \Cref{rem: important}), $\mathbf{g}$ admits a $cd$-index $\Phi_{\mathbf{g}}(c,d)$ which, using linearity, can be computed via the formula 
    \begin{align*}
        \Phi_{\mathbf{g}}(c,d)&=\sum_{i=0}^{n-1} (h_i(P)-\mathbf{w}_i)\check\Phi^n_i(c,d)
    \end{align*}
     by \cite[Theorem 3.1]{Stanley94}.  
  We now consider the difference $\mathbf{d}=(\mathbf{d}_S)_{S\subseteq [n]}\coloneqq(f'_S(P))_{S\subseteq [n]}-(\mathbf{g}_S)_{S\subseteq [n]}$, that is for $S=\{s_1<\cdots <s_k\}$ we have
    \[
        \mathbf{d}_S=(\chi(P)-\chi(\mathbb{S}^{n-1}))\left[\binom{n}{s_k}\binom{s_k}{s_1,s_2-s_1,\dots,s_k-s_{k-1}}\sum_{\lfloor\frac{n}{2}\rfloor+1\leq j \leq s_k}(-1)^{j}\binom{s_k}{j}+\delta_{S,\{n\}}\right].
    \]
    Applying \Cref{lem: identity hard} to the expression in the square brackets we obtain
    \[
        \mathbf{d}_S=(\chi(P)-\chi(\mathbb{S}^{n-1}))\sum_{i=\lfloor\frac{n}{2}\rfloor+1}^{n-1}(-1)^i\binom{n}{i}f_S(\check\Phi_i^n).
    \]
    Denoting by $\Phi_{\mathbf{d}}(c,d)$ the $cd$-index of $\mathbf{d}$ (which exists by linearity), the previous identity implies
    \begin{equation}\label{eq:cd for d}
        \Phi_{\mathbf{d}}(c,d)=(\chi(P)-\chi(\mathbb{S}^{n-1}))\sum_{i=\lfloor\frac{n}{2}\rfloor+1}^{n-1}(-1)^i\binom{n}{i}\check\Phi^n_i(c,d).  
    \end{equation}
    Using the definition of $\mathbf{d}$ and linearity we can compute $\Phi_{P}(c,d)$ via
    \begin{align*}
        \Phi_P(c,d)=&\Phi_{\mathbf{d}}(c,d)+\Phi_{\mathbf{g}}(c,d)\\
        =&\Phi_{\mathbf{d}}(c,d)+\sum_{i=0}^{n-1}h_i(P)\check\Phi^n_i(c,d)-(\chi(P)-\chi(\mathbb{S}^{n-1}))\sum_{i=\lfloor\frac{n}{2}\rfloor+1}^{n-1}(-1)^i\binom{n}{i}\check\Phi^n_i(c,d)\\
        =&\Phi_{\mathbf{d}}(c,d)+\sum_{i=0}^{n-1}h_i(P)\check\Phi^n_i(c,d)-\Phi_{\mathbf{d}}(c,d)\\
    =&\sum_{i=0}^{n-1}h_i(P)\check\Phi^n_i(c,d),
    \end{align*}
    where the third equality follows from \eqref{eq:cd for d}. This finishes the proof.
\end{proof}

\begin{example}
    Let $T$ be the unique triangulation of $\mathbb{S}^1\times \mathbb{S}^1$ on $7$ vertices. Hence $P=P(T)$ is a semi-Eulerian simplicial poset. We have that $h(P)=(1,4,10,-1)$, $\chi(\mathbb{S}^{2})-\chi(P)=2$ and
    \[
        \chi_P(a,b)=a^3+14a^2b+21aba+7ba^2+42ab^2+42bab+42b^2a+84b^3
    \]
    and $\chi'_P(a,b)=\chi_P(a,b)+2a^2b$. We compute
    \[
        \Psi_P(a,b)=\chi_P(a-b,b)=a^3+13a^2b+20aba+6ba^2+8ab^2+22bab+15b^2a-b^3
    \]
    and
    \begin{align*}
         \chi'_P(a-b,b)&=\Psi_P(a,b)+(2a^2b-2ab^2-2bab+2b^3)\\
         &=a^3+15a^2b+20aba+6ba^2+6ab^2+20bab+15b^2a+b^3,
    \end{align*}
    which gives
    \[
        \Phi_P(c,d)=c^3+14cd+5dc.
    \]
    Using \Cref{thm:CDvsH} we obtain the same $cd$-polynomial computing
    \begin{align*}
        h_0(P)\check\Phi^3_0(c,d)+h_1(P)\check\Phi^3_1(c,d)+h_2(P)\check\Phi^3_2(c,d) &= c^3+dc + 4(cd+dc)+10cd\\
        &=c^3+14cd+5dc.
    \end{align*}
    Observe that all coefficients of the $cd$-index are non-negative in this example. This is no coincidence as we will see in \Cref{sec: nonnegative}.
\end{example}
\begin{remark}
    \Cref{thm:CDvsH} allows to write the coefficients of the $cd$-index of a semi-Eulerian simplicial poset $P$ as linear combinations of the numbers $h_i(P)$ in the range where we can explicitly compute the polynomials $\check\Phi^n_i(c,d)$. For instance, if we write $h_i=h_i(P)$, we have:
    \begin{table}[h]
        \centering
        \begin{tabular}{l|l}
             $n$ & $\Phi_P(c,d)$ \\\hline
             $3$ & $h_0c^3+(h_1+h_2)cd+(h_0+h_1)dc$\\
             $4$ & $h_0c^4+(h_1+h_2+h_3)c^2d+(2h_0+2h_1+h_2)cdc+(h_0+h_1)dc^2+(h_1+2h_2+h_3)d^2$\\
             $5$ & $h_0c^5+(h_1+h_2+h_3+h_4)c^4d+(3h_0+3h_1+2h_2+h_3)c^2dc+(5h_0+3h_1+h_2)cdc^2+(3h_0+h_1)dc^3$\\
             &$+(2h_1+4h_2+4h_3+2h_4)cd^2+(2h_1+3h_2+3h_3+2h_4)dcd+(4h_0+4h_1+3h_2+h_3)d^2c$\\
             $6$ & $h_0c^6+(h_1+h_2+h_3+h_4+h_5)c^4d+(4h_0+4h_1+3h_2+2h_3+h_4)c^3dc+(9h_0+6h_1+3h_2+h_3)c^2dc^2$\\
             & $+(9h_0+4h_1+h_2)cdc^3+(4h_0+h_1)dc^4
             +(3h_1+6h_2+7h_3+6h_4+3h_5)c^2d^2$\\
             &$+(5h_1+8h_2+9h_3+8h_4+5h_5)cdcd+(3h_1+4h_2+4h_3+4h_4+3h_5)dc^2d$\\
             &$+(12h_0+12h_1+10h_2+6h_3+2h_4)cd^2c+(10h_0+10h_1+8h_2+5h_3+2h_4)dcdc$\\
             &$+(12h_0+8h_1+4h_2+h_3)d^2c^2+(4h_1+8h_2+10h_3+8h_4+4h_5)d^3$
        \end{tabular}
        \caption{$cd$-indices of semi-Eulerian simplician posets of rank $n+1$, with $3\leq n\leq 6$. Coefficients are linear combinations of the $h$-numbers.}
        \label{tab:my_label}
    \end{table}
    
    Interestingly, it appears from the examples we computed that the sequence $([w]\check\Phi^n_i(c,d) : 0\leq i\leq n-1)$ is unimodal for any monomial $w$. We currently do not have an explanation for this phenomenon, even though it is very tempting to conjecture that it holds for every $n$.
\end{remark}

\section{Bounding the coefficients of $\check\Phi_i^n$ via André permutations}

One of the steps to  show that the $cd$-index of a semi-Eulerian poset is always non-negative, consists in proving lower bounds for the coefficients of certain monomials $w$ in $\check\Phi^n_j(c,d)$. These bounds will be in terms of the number of occurrences of the letter $d$ in $w$. 

In the following, we use $[w]\Phi_P(c,d)$ to denote the coefficient of $w$ in $\Phi_P(c,d)$. The main result of this section is the following:

\begin{theorem}\label{thm: Andre}
    Let $n\geq 4$ and let $w$ be a $cd$-monomial with $m$ occurrences of the variable $d$. Then
    \begin{itemize}
        \item[(1)] If $\deg(w)=n-1$, then $[wc]\check\Phi_0^n\geq 2^m$.
        \item[(2)] If $n\geq 5$, $1\leq j \leq n-1$ and $\deg(w)=n-2$, then $[wd]\check\Phi_j^n\geq 2^m$.
    \end{itemize}
\end{theorem}

The proof of this theorem will rely on two different techniques. Namely, the coefficients of $\check\Phi_j^n(c,d)$ are known to count a particular class of permutations, so-called \emph{Andr\'e permutations} (with some additional properties) \cite{Het}. Second, there is a specific derivation on $\mathbb Z\langle c,d\rangle$, the polynomial ring in non-commuting variables $c$ and $d$ that connects $\check\Phi$-polynomials for different values of $n$ and $j$ \cite{Ehrenborg-Readdy}. We now make this more precise. 

We start by recalling the relevant definitions. For notation and definitions concerning André permutations we follow  \cite{Het}. 

Let $X$ be a totally ordered finite set. A \emph{permutation} of $X$ is a bijection  $\pi:[|X|]\to X$. A \emph{descent} of a permutation $\pi$ is a number $1\le i \le |X|-1$ such that $\pi(i)>\pi(i+1)$.  A permutation is said to have a \emph{double descent} if there is $1\le i \le |X|-2$ such that $i$ and $i+1$ are descents. 

\begin{definition}
A permutation $\pi$ of an ordered set $X$ of size $n$ is called an \emph{André permutation} if the following two conditions hold:
\begin{enumerate}
    \item[(1)] $\pi$ has no double descent. 
    \item[(2)] For every pair of numbers  $2\leq j<j'\leq n-1$ with $\pi(j-1)= \max\{\pi(j-1), \pi(j), \pi(j'-1), \pi(j')\}$ and $\pi(j')= \min\{\pi(j-1), \pi(j), \pi(j'-1), \pi(j')\}$, there exists  $j<j''<j'$ with $\pi(j'')<\pi(j')$. 
\end{enumerate}
\end{definition}
Given an André permutation $\pi$ of a set $X$ with $|X|=n$, one can associate a degree $n$ $cd$-monomial to $\pi$ as follows, where we write $\epsilon$ for the \emph{empty} permutation, i.e., $X=\emptyset$: 
\begin{itemize}
\item[(1)] If $n=0$, then $W(\epsilon)=1$ and if $n=1$, then $W(\pi)=c$ for the unique permutation $\pi$ of $W$. 
\item[(2)] If $n\geq2$, then 
\[
W(\pi)= \begin{cases}
    W(\pi|_{[n-2]})d, \quad\mbox{ if } n-1 \mbox{ is a descent}\\
    W(\pi|_{[n-1]})c, \quad  \mbox{ otherwise}.
\end{cases}
\]
 \end{itemize}
 We will refer to $W(\pi)$ as \emph{$cd$-type} of $\pi$. 
Our interest in André permutations is grounded in the following result of Hetyei (see \cite[Theorem 2]{Het}) that was originally conjectured by Stanley \cite{}. 

\begin{proposition}\label{prop: coeffAndre}
    
    Let $0\le j < n$ be positive integers and let $w$ be a $cd$-monomial of degree $n$. Then $[w]\check{\Phi}^n_j$ is equal to the number of André permutations $\pi$ of $[n]$ with $W(\pi)=w$ and $\pi(n)=n-j$.
\end{proposition} 

Given the previous result, one way to prove \Cref{thm: Andre} is to show that there are at least $2^m$ André permutations with $m$ occurrences of $m$ if the image of the maximal element is fixed. Indeed, we will severely undercount the permutations of interest to simplify the procedure. The refined counts are expected to be complicated in general and not necessary for our purposes. Our arguments will make use of the following alternative characterization of André permutations (see \cite[Proprieté 3.6]{Foata} and \cite[Proposition 5.5]{Purtill}).

\begin{lemma}\label{lem: Andre recursion}
    Let $X$ be an ordered set of size $n$, $\pi:[n]\to X$ be a permutation and let $m$ be the preimage of the minimal element of $X$. Then $\pi$ is an André permutation if and only if $\pi|_{[m]}$ and $\pi|_{[m+1,n]}$ are André permutations. 
\end{lemma}

We now provide the proof of part (1) of \Cref{thm: Andre}. 

\begin{proposition}\label{lem: Phi^n_0}
    Let $n\ge 4$ and let $w$ be a $cd$-monomial of degree $n-1$ with $m$ occurrences of the variable $d$. Then, there  are at least $2^m$ André permutations $\pi$ with $W(\pi)=wc$ and $\pi(n)=n$. Consequently, 
    \[
    [wc]\check\Phi^n_{0}\ge 2^m.
    \]
\end{proposition}

\begin{proof}
    The proof is by induction on $n$ where we need the two base cases $n=4$ and $n=5$ since in the induction step we will need the two previous values of $n$. A direct computation shows
    \[\check\Phi^4_{0}= c^4+2dc^2+2cdc
    \]
    and
    \[
    \check\Phi^5_0=c^5+3dc^3+5cdc^2+3c^2dc+4d^2c
    \]
and the claim is easily verified. 

    For the induction step, assume that $n\ge 6$. Given a $cd$-monomial $w$ of degree $n-1$, we distinguish two cases depending on whether $w$ starts with a $c$ or $d$: 

 {\sf Case 1:} $w = cw'$ for a degree $n-2$ $cd$-monomial $w'$. André permutations of type $wc$ cannot have a descent at position $1$, which is, in particular, fulfilled if $1$ is mapped to $1$. We show the stronger statement that there are at least $2^m$ André permutations $\pi$ of type $w$ with $\pi(1)=1$ and $\pi(n)=n$. By \Cref{lem: Andre recursion}, these are in bijection with André permutations of $[n-1]$, mapping $n-1$ to $n-1$ and that are of type $w'$. Since $w$ and $w'$ have the same number of occurrences of the variable $d$, the claim follows from the induction hypothesis.
        
{\sf Case 2:} $w = dw'$  for a degree $n-3$ $cd$-monomial $w'$. Since $w$ starts with a $d$, any André permutation of type $wc$ has a descent at position $1$. We count André permutations $\pi$ of type $wc$ with $\pi(1)=k$ for a fixed $k\in \{2,\ldots,n-1\}$, $\pi(2)=1$ and $\pi(n)=n$. Again, by \Cref{lem: Andre recursion}, these are in bijection with André permutations of $[n-2]$ of type $w'c$ mapping $n-2$ to $n-2$ (for any fixed valued of $k$). Since the variable $d$ occurs $m-1$ times in $w'$, the induction hypothesis now implies
\[
[wc]\check\Phi^n_{0} \ge (n-2)\cdot [w'c]\check\Phi^{n-2}_0 \ge (n-2)2^{m-1}\ge 2^m,
\]
since $n\geq 6$. The second claim follows from \Cref{prop: coeffAndre}.
    
\end{proof}

The next proposition provides the proof of \Cref{thm: Andre} (2) for the case $j=1$.

\begin{proposition}\label{lemma: PhiLowBound}
    Let $n\ge 5$ and  let $w$ be a $cd$-monomial of degree $n-2$ with $m$ occurrences of the variable $d$ in $w$. Then there are at least $2^m$ André permutations $\pi$ of $[n]$  with $W(\pi)=wd$ and $\pi(n)=n-1$. Consequently, 
    \[
    [wd]\check{\Phi}^n_1\geq 2^m.
    \]
\end{proposition}
We note that the previous proposition is not true for $n=4$, since $\check\Phi_4^1=d
c^2+2cdc+c^2d+d^2$. Hence, $[d^2]\check\Phi^1_4=1<2^1$. 

\begin{proof}

    The proof is by induction with the two base cases $n=5$ and $n=6$ that can be checked manually by computing $\check{\Phi}_5^1$ and $\check{\Phi}_6^1$. More precisely, we have:
\[
    \check{\Phi}^5_{1}=c^3d + 3c^2dc + 3cdc^2 + 2cd^2 + dc^3 + 2dcd + 4d^2c\]
    and
    \[
    \check{\Phi}^6_{1}=c^4d + 4c^3dc + 6c^2dc^2 + 3c^2d^2 + 4cdc^3 + 5cdcd + 12cd^2c + dc^4 + 3dc^2d + 10dcdc + 8d^2c^2 + 4d^3.
    \]
    
    To proceed let $n\ge 7$ and assume that the result is true for smaller values of $n$. Let $w$ be a $cd$-monomial of degree $n-2$ with $m$ occurrences of $d$. We distinguish two cases.
    
    {\sf Case 1:} $w=cw'$ for a degree $n-3$ $cd$-monomial $w'$. We show the stronger claim that there are at least $2^m$ André permutations $\pi$ of type $wd$ with $\pi(n)=n-1$ and $\pi(1)=1$. Using \Cref{lem: Andre recursion}, it follows that these are in bijection to André permutations of $[n-1]$ ending in $n-2$ (and starting in an arbitrary value) of type $w'd$. Since in $w'$ the variable $d$ occurs  $m$ times, we conclude by induction that there at least $2^m$ many such permutations.
        
     {\sf Case 2:}  $w=dw'$ for a degree $n-4$ $cd$-monomial $w'$. Since $w$ starts with a $d$, any André permutation of type $wd$ has a descent at position $1$. We count André permutations $\pi$ of type $wc$ with $\pi(1)=k$ for a fixed $k\in \{2,\ldots,n-2\}$, $\pi(2)=1$ and $\pi(n)=n-1$. By \Cref{lem: Andre recursion}, these are in bijection with André permutations of type $w'd$ mapping $n-2$ to $n-3$ (for any fixed valued of $k$). Since the variable $d$ occurs $m-1$ times in $w'$, the induction hypothesis now implies
\[
[wc]\check\Phi^n_{0} \ge (n-2)\cdot [w'c]\check\Phi^{n-2}_0 \ge (n-2)2^{m-1}\ge 2^m,
\]
since $n\geq 6$. The second claim follows from \Cref{prop: coeffAndre}.
    
\end{proof}

We want to emphasize that the approach above also works (with slight changes) for other values of $j$. However, we were not able to make this method work if $j=n-1$. To avoid further technicalities, we decided to only use it for $j=1$ since all other cases can be handled uniformly with a different approach which we now describe.
In \cite{Ehrenborg-Readdy}, Ehrenborg and Readdy consider the unique derivation $G$ on $\mathbb{Z\langle c,d\rangle}$ satisfying $G(c)=d$ and $G(d)=cd$. The main reason why this derivation is of interest to us is the following: 

\begin{theorem}\label{thm: Ehrenborg Readdy recursion}\cite[Theorem 8.1]{Ehrenborg-Readdy}
    Let $0\le j < n$ be positive integers. Then $G(\check\Phi_j^n)= \check\Phi_{j+1}^{n+1}$.  
\end{theorem}
This identity, together with the fact that $G$ is a derivation, i.e., it obeys Leibniz rule, will be the second important trick to complete the proof of \Cref{thm: Andre}.
\begin{remark}
    Notice that the bounds here are very far from tight. Yet the argument somehow needs both parts, one could take care of different values of $j$. Many of the choices along the proof are arbitrary in order to make sure that contribution is large enough. In the permutation case we could allow the value of 1 to be anything, or in the second part look at permutations starting with $k1$. In the derivation proof (which would also work for any value of $j\ge 2$) we need to find enough monomials to apply $G$ and produce coefficients.
\end{remark}

The next statement completes the proof of \Cref{thm: Andre} (2).
\begin{proposition}
Let $n\geq 5$, $2\leq j\leq n-1$ and $w$ be a $cd$-monomial of degree $n-2$ with $m$ occurrences of the variable $d$. Then there are at least $2^m$ André permutations $\pi$ of $[n]$  with $W(\pi)=wd$ and $\pi(n)=n-j$. Consequently, 
    \[
    [wd]\check{\Phi}^n_j\geq 2^m.
    \]
\end{proposition}
\begin{proof}
As in the proof of \Cref{lemma: PhiLowBound} we proceed by induction on $n$ and first verify the two base cases $n=5$ and $n=6$ by computing $\check\Phi^j_5$ and $\check\Phi_j^6$, explicitly. We have
    \begin{align*}
    \check{\Phi}^5_{2}=&c^3d + 2c^2dc + cdc^2 + 4cd^2 + 3dcd + 3d^2c\\
    \check{\Phi}^5_{3}=&c^3d + c^2dc + 4cd^2 + 3dcd + d^2c\\
    \check{\Phi}^5_{4}=&c^3d + 2cd^2 + 2dcd
    \end{align*}
    and
    \begin{align*}
    \check{\Phi}^6_{2}=&c^4d + 3c^3dc + 3c^2dc^2 + 6c^2d^2 + cdc^3 + 8cdcd + 10cd^2c + 4dc^2d + 8dcdc + 4d^2c^2 + 8d^3\\
    \check{\Phi}^6_{3}=&c^4d + 2c^3dc + c^2dc^2 + 7c^2d^2 + 9cdcd + 6cd^2c + 4dc^2d + 5dcdc + d^2c^2 + 10d^3\\
    \check{\Phi}^6_{4}=&c^4d + c^3dc + 6c^2d^2 + 8cdcd + 2cd^2c + 4dc^2d + 2dcdc + 8d^3\\
    \check{\Phi}^6_{5}=&c^4d + 3c^2d^2 + 5cdcd + 3dc^2d + 4d^3.
    \end{align*}
    To proceed let $n\ge 7$ and assume that the result is true for smaller values of $n$. 
    
    The idea is to use \Cref{thm: Ehrenborg Readdy recursion} together with the fact that the derivation $G$ preserves positive coefficients. We will derive  estimates  for the coefficient of a certain $cd$-monomial  $wd$ by looking at the lower degree monomials whose derivation includes $wd$. Since the polynomials $\check\Phi^n_j$ only have non-negative coefficents, no further cancelling can occur. We distinguish two cases, depending on if the variable $c$ occurs in $w$ or does not. 

    {\sf Case 1:} $w=w'cd^{t-1}$ for some $cd$-monomial $w'$ and $t\geq 1$. 
    Since, by Leibniz's rule, $wd=w'G(d)d^{t-1}$, we have $[wd]G(w'd^t)\geq 1$ and it follows from \Cref{thm: Ehrenborg Readdy recursion} that $[wd]\check\Phi_j^n\geq [w'd^t]\check\Phi_{j-1}^{n-1}$. Since $w'd^t$ contains the variable $d$ $m$ times, the claim follows from the induction hypothesis for $j\geq 3$ and it follows from \Cref{lemma: PhiLowBound} for $j=2$. 

Note that, we are always in Case 1, if $n$ is odd. Assume that $n$ is even. We now have to consider an additional case.

{\sf Case 2:} $w=d^{m}$ with $m=\frac{n-2}{2}$. Since $n\geq 7$ and $n$ is even, we have $m\geq 3$. Applying Leibniz's rule we have that $[d^{m+1}]G(cd^{m})\ge 1$ and $[d^{m+1}]G(dcd^{m-1})\ge 1$. Using \Cref{thm: Ehrenborg Readdy recursion} we conclude that 
\[[d^{m+1}]\check{\Phi}^{n}_{j} \ge [cd^{m}]\hat{\Phi}^{n-1}_{j-1}+ [dcd^{m-1}]\hat{\Phi}^{n-1}_{j-1}.
\]
For $j\geq 3$, each summand on the right-hand side of the above equation is at least $2^{m-1}$ by induction. For $j=2$, the same is true due to \Cref{lemma: PhiLowBound}. This finishes the proof.
\end{proof}

\begin{remark}\label{rem:j equals 1}
We note that the case $j=1$ cannot be shown directly using the recursion due to Ehrenborg and Readdy. The issue here is that in $\check\Phi^n_0$ all monomials with positive coefficients end with a $c$ since the coefficients count André permutations ending in $n$. In particular, there is only one monomial in $\check\Phi_0^n$ whose derivation contains a specific monomial $wd$. Indeed, $[wd]G(wc)\geq 1$. In order to conclude as in the previous proof, one would need the bound $[wc]\check\Phi_0^n\geq 2^{m+1}$, if the variable $d$ occurs $m$ times in $w$. However, we were only able to show that 
$[wc]\check\Phi_0^n\geq 2^{m}$ (see \Cref{thm: Andre} (1)).
\end{remark}

\section{$P$-polynomials and their coefficients}\label{sec : P poly}
In this section, we will provide the last result needed to show that non-negativity of the $cd$-index of a Buchsbaum semi-Eulerian simplicial poset. For this aim, we need to introduce the following family of polynomials:
\begin{definition}
 For a positive $n$ and $0\le j \le n-1$, we let 
\[
P_{j}^n(c,d)= \sum_{i=j}^{n-1}(-1)^{i-j}\binom{n}{i}\check\Phi_i^n(c,d)\in \mathbb{Z}\langle c,d\rangle
\]
\end{definition}
For instance we have that $P^2_0=c^2-2d$ and $P^2_1=2d$. We include a list of all polynomials $P^n_j$, for $3\leq n\leq 6$.
\begin{table}[h]
    \centering
    \begin{tabular}{l|l|l}
     $P^3_{0} = c^3 \underline{- 2dc}$ & $P^4_{0} = c^4 \underline{- 2c^2d} \underline{- 2dc^2} + 4d^2$ & $P^5_{0}=c^5 \underline{- 2c^2dc} \underline{- 2dc^3} + 4d^2c$\\
     $P^3_{1} = 3dc$ & $P^4_{1} =2c^2d + 2cdc + 4dc^2 \underline{- 4d^2}$ & $P^5_{1}=5c^2dc + 5cdc^2 + 5dc^3$\\
     $P^3_{2} = 3cd$ & $P^4_{2} =  2c^2d + 6cdc + 8d^2$ & $P^5_{2}=5c^3d + 10c^2dc + 10cdc^2 + 10cd^2 + 10dcd + 20d^2c$\\
      & $P^4_{3} = 4c^2d + 4d^2$ & $P^5_{3}= 5c^3d + 10c^2dc + 30cd^2 + 20dcd + 10d^2c$\\
    & & $P^5_{4}=5c^3d + 10cd^2 + 10dcd$
    \end{tabular}
    \label{tab:my_label}
\end{table}
\begin{table}[h]
    \centering
    \begin{tabular}{l}
       $P^6_0=c^6 \underline{- 2c^4d} \underline{- 2c^2dc^2} + 4c^2d^2 \underline{- 2dc^4} + 4dc^2d + 4d^2c^2 \underline{- 8d^3}$\\
       $P^6_1=2c^4d + 4c^3dc + 11c^2dc^2 \underline{- 4c^2d^2} + 9cdc^3 + 12cd^2c + 6dc^4 \underline{- 4dc^2d} + 10dcdc + 8d^2c^2 + 8d^3$\\
       $P^6_2=4c^4d + 20c^3dc + 25c^2dc^2 + 22c^2d^2 + 15cdc^3 + 30cdcd + 60cd^2c + 22dc^2d + 50dcdc + 40d^2c^2 + 16d^3$\\
       $P^6_3=11c^4d + 25c^3dc + 20c^2dc^2 + 68c^2d^2 + 90cdcd + 90cd^2c + 38dc^2d + 70dcdc + 20d^2c^2 + 104d^3$\\
       $P^6_4=9c^4d + 15c^3dc + 72c^2d^2 + 90cdcd + 30cd^2c + 42dc^2d + 30dcdc + 96d^3$\\
       $P^6_5=6c^4d + 18c^2d^2 + 30cdcd + 18dc^2d + 24d^3$
    \end{tabular}
    \caption{The polynomials $P^n_j$, for $n\leq 6$. Monomials with negative coefficient are underlined.}
    \label{tab:my_label}
\end{table}

The following statement is the main result of this section: 

\begin{theorem}\label{thm: PnjPositivity}
    Let $n\ge 3$ be a positive integer and let $2\leq j\leq n-1$. Let $w$ be a $cd$-monomial with $m$ occurrences of the variable $d$. Then:
    \begin{itemize}
        \item[(1)] If $\deg(w)=n-1$, then $[wc]P_j^n\geq \binom{n-1}{j-1}2^m$.
        \item[(2)] If $\deg(w)=n-2$, then $[wd]P_j^n\geq (\binom{n-1}{j-1}-1)2^m$.
    \end{itemize}
    In particular, all coefficients of $P_{j}^n(c,d)$ are non-negative. 
\end{theorem}

We will see in \Cref{sec: nonnegative} that \Cref{thm: PnjPositivity} implies non-negativity of $\Phi_P(c,d)$ for any Buchsbaum semi-Eulerian simplicial poset $P$. This might come a bit as a surprise at first since the polynomials $P_{j}^n(c,d)$ are independent of the considered poset. 

We note that \Cref{thm: PnjPositivity} says nothing about the polynomials  $P_{0}^n(c,d)$ and $P_{1}^n(c,d)$ and indeed these polynomials have quite a few negative coefficients. The first steps towards proving \Cref{thm: PnjPositivity}  is to understand which coefficients of $P_0^n(c,d)$ are $P_1^n(c,d)$ are indeed negative. For $P_0^n(c,d)$ the precise behavior of its coefficients can be read off directly from the following simple closed formula for $P_{0}^n(c,d)$.
 
 \begin{lemma}\label{lem: Pn0}
     For a positive integer $n\ge 3$  we have
     \begin{equation} \label{eq: Pn0}P_{0}^n(c,d)= \begin{cases} (c^2-2d)^{\frac{n}{2}} & \text{ if } n \text{ is even,}\\
    (c^2-2d)^{\frac{n-1}{2}}c & \text{ if } n \text{ is odd}.
\end{cases}\end{equation}
 \end{lemma}

 \begin{proof}
 Throughout this proof, we use the same notation as in  \Cref{lem: identity hard}. 
      We prove the claim by looking at the flag $f$-vectors that are associated to both sides of the identity. Setting $\chi(a,b)=a^2$, we observe that 
      \[
            \chi(a-b,b)=(a-b)^2=(a+b)^2-2(ab+ba)=c^2-2d.
        \]
        It follows that 
        \[
        \chi^k(a-b,b)= (c^2-2d)^k \quad \mbox{and} \quad  (a+b)\chi^k(a,b)=(c^2-2d)^kc
        \]
        for $k\geq 1$. Hence, \eqref{eq: Pn0} is equivalent to the following statement
        \[
        f_S(P_{0}^n)=\begin{cases}
            \delta_{S,\emptyset}, \quad &\mbox{ if } n \mbox{ is even},\\
            \delta_{S,\emptyset}+\delta_{S,\{n\}} , \quad &\mbox{ if } n \mbox{ is odd}
            \end{cases}.
            \]
            We distinguish four cases.
            
       {\sf Case 1:} Assume that $S=\emptyset$. Using that, by definition of $\check\Phi_i^n$, we have $f_{\emptyset}(\check\Phi_i^n(c,d))=0$ for $1\leq i\leq n-1$ and $f_{\emptyset}(\check\Phi_0^n(c,d))=0$, we obtain
        \begin{equation*}
            f_{\emptyset}(P_{0}^n) = \sum_{i=0}^{n-1}(-1)^{i}\binom{n}{i}f_{\emptyset}(\check\Phi_i^n(c,d))=f_{\emptyset}(\check\Phi_0^n(c,d))=1.
        \end{equation*}
        {\sf Case 2:} Assume that $s_k<n$. Using \eqref{eq: identity 1} and \eqref{eq: identity 2} as in Case 1 of the proof of \Cref{lem: identity hard} we obtain
        \begin{equation*}
            f_S(P_{0}^n) = \sum_{i=0}^{n-1}(-1)^{i}\binom{n}{i}f_S(\check\Phi_i^n(c,d))=\binom{n}{s_1,s_2-s_2,\dots,n-s_k}\sum_{i=0}^{n-1}(-1)^{i}\binom{s_k}{i} =0.
        \end{equation*}
        {\sf Case 3:} Assume that $s_k=n$, and $S\neq \{n\}$. The same computation as in Case 2 of the proof of \Cref{lem: identity hard} shows that
        \begin{equation*}
            f_S(P_{0}^n) = \sum_{i=0}^{n-1}(-1)^{i}\binom{n}{i}f_S(\check\Phi_i^n(c,d))=\binom{n}{s_1,s_2-s_2,\dots,n-s_{k-1}}\sum_{i=0}^{n}(-1)^{i}\binom{n}{i} =0.
        \end{equation*}
        {\sf Case 4:} Assume that $S= \{n\}$. Being $\Lambda^n_0$ the semisuspension of a simplex we have that $f_{\{n\}}(\check\Phi_0^n(c,d))=f_{\{n\}}(P(\Lambda^n_0))=2$. Moreover, $f_{\{n\}}(\check\Phi_i^n(c,d))=f_{\{n\}}(P(\Lambda^n_i))-f_{\{n\}}(P(\Lambda^n_{i-1}))=1$ for $1\leq i\leq n-1$. This implies
        \begin{equation*}
            f_{\{n\}}(P_{0}^n) = \sum_{i=0}^{n-1}(-1)^{i}\binom{n}{i}f_{\{n\}}(\check\Phi_i^n(c,d))=2+\sum_{i=1}^{n-1}(-1)^{i}\binom{n}{i}=\begin{cases}
                0, & \text{if $n$ even}\\
                1, & \text{if $n$ odd},
            \end{cases}
        \end{equation*}
        and the claim follows.
 \end{proof}

The following recursion will be crucial to prove \Cref{thm: PnjPositivity} since it enables us to use induction.

\begin{lemma}\label{lem:recurrence}
Let $n\ge 5$ be a positive integer, $2\le j \le n-2$ and let $w$ be a $cd$-monomial. 
\begin{enumerate}
\item[(1)] If $w$ has degree $n-1$, then 
\begin{equation}\label{eq: wc final}
[wc]P_{j}^n = [w]P_{j}^{n-1} + \binom{n-1}{j-1}[wc]\check{\Phi}_{j}^n.
\end{equation}
\item[(2)] If $w$ has degree $n-2$, then 
\begin{equation}\label{eq: wd final}
    [wd]P_{j}^n=[wc]((-1)^{n+j+1} P_{0}^{n-1} - P_{j}^{n-1}+P_{n-j}^{n-1})+\binom{n-1}{j-1}[wd] \check{\Phi}_j^n.
\end{equation}
\end{enumerate}
\end{lemma}
\begin{proof}
In this proof, we will strongly use the following two identities, that are true for $0\leq i\leq n-1$, from \cite[Remark 4.3]{Novik:cd}:
\begin{equation}\label{eq:wc-monomials}
   [wc](\check{\Phi}_i^n -\check{\Phi}_{i+1}^n)=[w]\check{\Phi}_i^{n-1}
\end{equation}
and
\begin{equation}\label{eq:wd-monomials}
   [wd](\check{\Phi}_i^n -\check{\Phi}_{i+1}^n)=[wc](\check{\Phi}_{n-1-i}^{n-1}-\check{\Phi}_i^{n-1}).
\end{equation}
Using the usual binomial recursion, can rewrite $P^n_j$ as follows:
\begin{equation}\label{eq: rewriting}
P_j^n=(-1)^j\sum_{i=j}^{n-1}\binom{n-1}{i}(-1)^i(\check{\Phi}_i^n-\check{\Phi}_{i+1}^{n})+\binom{n-1}{j-1}\check{\Phi}_j^n.
\end{equation}
The recursion in (1) follows from \eqref{eq: rewriting} and \eqref{eq:wc-monomials} as follows:
\begin{align*} 
    [wc]P_{j}^n=&(-1)^j\sum_{i=j}^{n-1}\binom{n-1}{i}(-1)^i[w]\check{\Phi}_i^{n-1}+\binom{n-1}{j-1}[wc]\check{\Phi}_j^n\\ \nonumber
    =&[w]((-1)^j\sum_{i=j}^{n-2}\binom{n-1}{i}(-1)^i\check{\Phi}_i^{n-1})+\binom{n-1}{j-1}[wc]\check{\Phi}_j^n\\ 
    =&[w]P_{j}^{n-1}+\binom{n-1}{j-1}[wc]\check{\Phi}_j^n.
\end{align*}

The proof of \eqref{eq: wd final} is more involved and will require a case distinction according to the parity of $n$. Using  \eqref{eq: rewriting} and \eqref{eq:wd-monomials} we get the following expression for the coefficient of $wd$:
\begin{align} \label{eq: wd-firstTrick}
 \nonumber [wd]P_{j}^n=&(-1)^j\sum_{i=j}^{n-1}\binom{n-1}{i}(-1)^i[wc](\check{\Phi}_{n-1-i}^{n-1}-\check{\Phi}_i^{n-1})+\binom{n-1}{j-1}[wd]\check{\Phi}_j^n\\ \nonumber
    =&(-1)^j\sum_{i=j}^{n-1-j}\left((-1)^{i+1}\binom{n-1}{i}+(-1)^{n-i-1}\binom{n-1}{n-i-1}\right)[wc]\check{\Phi}_{i}^{n-1}  +\binom{n-1}{j-1}[wd]\check{\Phi}_j^n\\ 
    &+(-1)^j\sum_{i=0}^{j-1}\binom{n-1}{n-1-i}(-1)^{n-1-i}[wc]\check{\Phi}_{i}^{n-1}  +(-1)^j\sum_{i=n-j}^{n-1}\binom{n-1}{i}(-1)^{i+1}[wc]\check{\Phi}_{i}^{n-1}.
\end{align}
We now analyze the right-hand side of \eqref{eq: wd-firstTrick}. The second sum can be rewritten as follows:  
\begin{align*}
    (-1)^j\sum_{i=0}^{j-1}\binom{n-1}{n-1-i}(-1)^{n-1-i}[wc]\check{\Phi}_{i}^{n-1}=& 
    (-1)^{n-1+}j\sum_{i=0}^{j-1}\binom{n-1}{i}(-1)^{i}[wc]\check{\Phi}_{i}^{n-1}\\
    =&(-1)^{n-1+j}[wc]\left(P_{0}^{n-1}-\sum_{i=j}^{n-1}\binom{n-1}{i}(-1)^{i}\check{\Phi}_{i}^{n-1}\right)\\
    =&(-1)^{n-1+j}[wc]\left(P_{0}^{n-1}-(-1)^jP_{j}^{n-1}\right) 
\end{align*}
We also note that the third sum equals $(-1)^{n-1}[wc]P_{n-j}^{n-1}$. 
To handle the first sum, assume that $n$ is odd. Using the symmetry of the binomial coefficients and that $(-1)^{i+1} =(-1)^{n-i}$ in this case, we see that the first sum on the right-hand side of \eqref{eq: wd-firstTrick} vanishes.  If $n$ is even, we obtain $(-1)^{i+1}\binom{n-1}{i}+(-1)^{n-i-1}\binom{n-1}{n-i-1}=(-1)^{i+1}2\binom{n-1}{i}$ for the coefficients in the first sum and we can rewrite this sum as follows:
\begin{align*}
    &(-1)^j\sum_{i=j}^{n-1-j}\left((-1)^{i+1}\binom{n-1}{i}+(-1)^{n-i-1}\binom{n-1}{n-i-1}\right)[wc]\check{\Phi}_{i}^{n-1} \\
    =&(-1)^j\sum_{i=j}^{n-1-j}(-1)^{i+1}2\binom{n-1}{i}[wc]\check{\Phi}_{i}^{n-1} \\
    =&-2[wc]P_j^{n-1}+2[wc]P_{n-j}^{n-1}.
\end{align*}

Substituting the derived expressions in \eqref{eq: wd-firstTrick} shows the desired recursion for every $n$.
\end{proof}

The crucial step towards proving \Cref{thm: PnjPositivity} is to provide bounds for the coefficients of the polynomials $P_j^n$. 
In doing so, we will treat separately the cases $j=2$ and $j\geq 3$. We start by providing an exact expression for the coefficients of monomials of $P^n_j$ ending in $d$.

\begin{lemma}\label{lem: Pn2}
    Let $n\geq 5$ be a positive integer and let $w$  be a $cd$-monomial of degree $n-2$ with $m$ occurrences of the variable $d$. Then:
    \begin{itemize}
    \item[(1)] If $n$ is even and odd, respectively, and all powers of $c$ in $wc$ and $w$, respectively, are even, then
    \begin{equation} \label{eq: lemma equality d,2} 
    [wd]P_{2}^n = (n-1)[wd]\check\Phi^n_1 +  [wc]\check\Phi_0^{n-1}+(1+(-1)^{n})(-1)^{m+1}2^m.
    \end{equation}
    \item[(2)] If none of the cases in (1) holds, we have
    \begin{equation} \label{eq: lemma equality d,2} 
    [wd]P_{2}^n = (n-1)[wd]\check\Phi^n_1 +  [wc]\check\Phi_0^{n-1}.
    \end{equation}
    \end{itemize}
    In particular, in both cases, one has
    \[
     [wd]P_{2}^n \geq (n-1)[wd]\check\Phi^n_1 +  [wc]\check\Phi_0^{n-1}-2^{m+1}.
    \]
     \end{lemma}

\begin{proof}
Using \Cref{eq: wd final}, $[wd]P_2^n$ can be computed as follows:
\[
[wd]P_2^{n} = [wc]((-1)^{n+1}P_{0}^{n-1}-P_2^{n-1} + P_{n-2}^{n-1}) + (n-1)[wd]\check\Phi^n_2.
\]
Since $P_2^{n-1}=(n-1)\check{\Phi}_{1}^{n-1}-P_1^{n-1}$ and $P_{n-2}^{n-1}= (n-1)\check\Phi^{n-1}_{n-2}$, this yields
\begin{align*}
     [wd]P_2^{n}  = & (-1)^{n+1}[wc]P_0^{n-1} +(n-1)[wc](\check\Phi_{n-2}^{n-1}-\check\Phi^{n-1}_1) + [wc]P_1^{n-1}+(n-1)[wd]\check\Phi^n_2 \\ 
        = &  (-1)^{n+1}[wc]P_0^{n-1} +(n-1)[wd](\check\Phi_{1}^{n}-\check\Phi^{n}_2) + [wc]P_1^{n-1}+(n-1)[wd]\check\Phi^n_2 \\ 
        =& (-1)^{n+1}[wc]P_0^{n-1} +(n-1)[wd]\check\Phi^{n}_1 + [wc]P_1^{n-1}.
        \end{align*}
As $P_1^{n-1}=\check\Phi_0^{n-1}-P_0^{n-1}$, we obtain
\[
     [wd]P_2^{n} =((-1)^{n+1}-1)[wc]P_0^{n-1}+[wc]\check\Phi_0^{n-1}+(n-1)[wd]\check\Phi^{n}_2.
\]
If $n$ is even and odd, respectively and all powers of $c$ in $wc$ and $w$, respectively, are even, it follows from \Cref{lem: Pn0} that $[wc]P_0^{n-1}=(-2)^m$. Otherwise, we have $[wc]P_0^{n-1}=0$.
  This finishes the proof of (1) and (2). The ``In particular''-statement is immediate, since $(1+(-1)^{n})(-1)^{m+1}2^m\geq -2^{m+1}$.
\end{proof}

We can finally provide the proof of \Cref{thm: PnjPositivity}.\\

\begin{proof}[Proof of \Cref{thm: PnjPositivity}]
First let $j=n-1$. If $n=3$, we have $P_2^3(c,d)=3cd$ and $P_{3}^4(c,d)=4ccd + 4dd$ and the claim is easily verified. 
Now let $n\geq 5$.  Since any permutation of $[n]$ mapping $n$ to $1$ has a descent at position $n-1$, \Cref{prop: coeffAndre} implies that all monomials occuring in $P_{n-1}^n$ with a non-zero coefficient have to end with the variable $d$. Using that, by definition, $P_{n-1}^n= n \check\Phi^n_{n-1}$, it now follows from \Cref{thm: Andre} (2) that 
$[wd]P_{n-1}^n\geq n\cdot 2^m\geq (\binom{n-1}{n-2}-1)2^m$. This shows the claim for the case $j=n-1$ and any $n\geq 5$. 

 In the following, we prove the remaining cases, i.e., $2\le j \le n-2$, by induction on $n$, the base case being $n=4$. By the previous argument, it only remains to verify the claim for $P_{2}^4$. But, as $P_{2}^4(c,d)=2c^2d+6cdc+8d^2$, the claim can easily be confirmed by hand. 

Now let $n\geq 5$ and assume by induction that the statements are true for the coefficients of any $P_j^s$ with $3\leq s \leq n-1$ and all  $2\leq j\leq s-2$. For the induction step, let $\tilde{w}$ be a $cd$-monomial of degree $n$. First, suppose that $\tilde{w}$ ends with a $c$, i.e., $\tilde{w}=wc$ for a degree $n-1$ $cd$-monomial $w$. In this case, \Cref{lem:recurrence} (1) implies
\[
[wc]P_{j}^n = [w]P_{j}^{n-1} + \binom{n-1}{j-1}[wc]\check{\Phi}_{j}^n\geq [w]P_{j}^{n-1},
\]
as the second term is non-negative by  \Cref{prop: coeffAndre}. If $w$ ends with a $c$, the claim follows directly from the induction hypothesis. If $w$ ends with a $d$, the induction hypothesis implies that 
\[
[wc]P_{j}^n \geq [w]P_{j}^{n-1}\geq (\binom{n-2}{j-1}-1)2^{m-1}.
\]
Since $j\geq 1$ and $n\geq 5$, we have $\binom{n-2}{j-1}-1\geq 3-1$ and 
the claim follows also in this case.

Now assume that $\tilde{w}$ ends with a $d$, i.e., $\tilde{w}=wd$ for a degree $n-2$ $cd$-monomial $w$. Let $m$ be the number of occurrences of $d$ in $w$. In the following, we distinguish two cases.
\begin{description}
  \item[{\sf Case 1:} $j=2$] Combining \Cref{lem: Pn2} and \Cref{thm: Andre}, we infer that
  \[
   [wd]P_{2}^n \geq (n-1)[wd]\check\Phi^n_1 +  [wc]\check\Phi_0^{n-1}-2^{m+1}\geq (n-1)2^m+2^m-2^{m+1}=(n-2)2^m
  \]
since $n\geq 5$. As $n-2=\binom{n-1}{2-1}-1$, the claim follows. 

  \item[{\sf Case 2:} $j\geq 3$]  We start by showing the following claim
      \begin{equation}\label{eq:helper}
      \binom{n-2}{j-1}[wc]\check\Phi_{j}^{n-1}\leq [wc]P_{j}^{n-1}\leq \binom{n-2}{j-1}[wc]\check{\Phi}_{j-1}^{n-1}.
      \end{equation}
      Indeed, \Cref{lem:recurrence} combined with the induction hypothesis and \Cref{prop: coeffAndre} directly implies the first inequality. Using this inequality for $P^{n-1}_{j-1}$ and again \Cref{lem:recurrence}, we obtain
      \begin{align*}
        [wc] P_{j}^{n-1}=&\binom{n-1}{j-1}[wc]\check{\Phi}_{j-1}^{n-1}-[wc]P_{j-1}^{n-1}\\
        \leq & \left(\binom{n-1}{j-1}-\binom{n-2}{j-2}\right)[wc]\check{\Phi}_{j-1}^{n-1}=\binom{n-2}{j-1}[wc]\check{\Phi}_{j-1}^{n-1}, 
      \end{align*}
      which finishes the proof of \eqref{eq:helper}. 
      Applying the lower and upper bound in \eqref{eq:helper} to $P_{n-j}^{n-1}$ and $P_j^{n-1}$, respectively, and using \Cref{lem:recurrence}, we obtain
\begin{align*}
        [wd]P_{j}^n=&[wc]((-1)^{n+j+1} P_{0}^{n-1} - P_{j}^{n-1}+P_{n-j}^{n-1})+\binom{n-1}{j-1}[wd] \check{\Phi}_j^n\\
        \geq &[wc]\left((-1)^{n+j+1} P_{0}^{n-1} - \binom{n-2}{j-1}[wc]\check{\Phi}_{j-1}^{n-1}+\binom{n-2}{n-j-1}[wc]\check{\Phi}_{n-j}^{n-1}\right)+\binom{n-1}{j-1}[wd] \check{\Phi}_j^n\\
        =& (-1)^{n+j+1} [wc]P_{0}^{n-1} +\binom{n-2}{j-1}[wc]\left(\check{\Phi}_{n-j}^{n-1}-\check{\Phi}_{j-1}^{n-1}\right)+\binom{n-1}{j-1}[wd] \check{\Phi}_j^n.
    \end{align*}
 Using \eqref{eq: rewriting} and \Cref{thm: Andre} (2), the above yields
    \begin{align*}
        [wd]P_{j}^n\geq & (-1)^{n+j+1} [wc]P_{0}^{n-1} +\binom{n-2}{j-1}[wd]( \check{\Phi}_{j-1}^{n}-\check{\Phi}_{j}^{n} ) +\binom{n-1}{j-1}[wd] \check{\Phi}_j^n \\
        =& (-1)^{n+j+1} [wc]P_{0}^{n-1} +\binom{n-2}{j-1}[wd]\check{\Phi}_{j-1}^{n}+\binom{n-2}{j-2}[wd] \check{\Phi}_j^n\\
        \geq &(-1)^{n+j+1} [wc]P_{0}^{n-1}+\binom{n-2}{j-1} 2^m+ \binom{n-2}{j-2}2^m\\
        \geq&-2^m+\binom{n-1}{j-1}2^m= (\binom{n-1}{j-1}-1)2^m,
    \end{align*}
where the last inequality follows from \Cref{lem: Pn0}. This finishes the proof.
\end{description}
\end{proof}

\section{Non-negativity of the $cd$-index for semi-Eulerian Buchsbaum simplicial posets}\label{sec: nonnegative}
The goal of this section is to show our main result; namely, that the $cd$-index of any Buchsbaum semi-Eulerian simplicial poset has only non-negative coefficients. The proof consists of applying \Cref{thm:NovSwa}, which is due to Novik and Swartz, to the formula for the $cd$-index from \Cref{thm:CDvsH} and then using the non-negativity of the $P$-polynomials (\Cref{thm: PnjPositivity}).

\begin{theorem}\label{thm:BoundPhi}
Let $n\geq 3$  and let $P$ be a semi-Eulerian Buchsbaum simplicial poset over a field $\mathbb{F}$ that is of rank $n+1$. Let $w$ be a $cd$-monomial with $m$ occurrences of the variable $d$. Then:
\begin{itemize}
    \item[(1)] If $\deg(w)=n-1$, then $[wc]\Phi_P\geq 2^m \left(\sum_{j=2}^{n-1}\beta_{j-1}(\Delta(\overline{P});\mathbb{F}) \binom{n-1}{j-1}\right)$.
      \item[(1)] If $\deg(w)=n-2$, then $[wd]\Phi_P\geq 2^m \left(\sum_{j=2}^{n-1}\beta_{j-1}(\Delta(\overline{P});\mathbb{F}) (\binom{n-1}{j-1}-1)\right)$.
\end{itemize}
In particular, all coefficients of $\Phi_{P}(c,d)$ are non-negative.
\end{theorem}

\begin{proof}
Since the coefficients of $\check{\Phi}^n_i$ are non-negative for every $i$ (see \Cref{prop: coeffAndre}), the inequality in \Cref{thm:NovSwa} can be applied  to each $h_i(P)$, individually, in the expression for $\Phi_P(c,d)$ from \Cref{thm:CDvsH}. This yields, for any $cd$-monomial $w$
\begin{align*}
    [w]\Phi_{P}(c,d)&=\sum_{i=0}^{n-1} h_i(P)[w]\check\Phi_i^n(c,d)\\
    &\geq  \sum_{i=0}^{n-1}\left( \binom{n}{i}\sum_{j=0}^{i}(-1)^{i-j}\beta_{j-1}(\Delta(\overline{P});\mathbb{F})\right) [w]\check\Phi_i^n(c,d)\\
    &= \sum_{j=2}^{n-1}\beta_{j-1}(\Delta(\overline{P});\mathbb{F})\left( \sum_{i=j}^{n-1}\binom{n}{i}(-1)^{i-j}[w]\check\Phi_i^n(c,d)\right)\\
    &= \sum_{j=2}^{n-1}\beta_{j-1}(\Delta(\overline{P});\mathbb{F}) [w]P_{j}^n(c,d),
\end{align*}
where the last two summations start at $j=2$, as $\beta_{-1}(\Delta(\overline{P});\mathbb{F}) = \beta_{0}(\Delta(\overline{P}); \mathbb{F})= 0$. We conclude by applying \Cref{thm: PnjPositivity}.
\end{proof}

Since the face lattice of a simplicial manifold is a semi-Eulerian Buchsbaum poset, we obtain the following.

\begin{corollary}\label{cor : manifolds}
    Let $P$ be the face lattice of a simplicial manifold. Then $\Phi_{P}(c,d)$ has nonnegative coefficients.
\end{corollary}

If $P$ is the face lattice of a $(2k-1)$-dimensional simplicial manifold, the non-negativity of some of the coefficients of $\Phi_{P}(c,d)$ has been proven by Novik in \cite[Theorem 2.1]{Novik:cd}. More precisely, she showed that all monomials that are either of the form $wdc^i$ for $i$ odd, or of the form $c^{2k-i-2}dc^i$ for $i$ even, occur with a non-negative coefficient in $\Phi_P(c,d)$.  \Cref{cor : manifolds} completes Novik's result and extends non-negativity to even dimensional manifolds, thereby providing a positive answer to one of the questions in \cite[Section 6]{Novik:cd}.

\begin{remark}
    Since each coefficient of $\Phi_{P}(c,d)$ can be expressed as a linear combination of entries of the flag $h$-vector,  it is natural to ask how the non-negativity of $\Phi_P(c,d)$ (see \Cref{thm:BoundPhi}) affects the flag $h$-vector of $P$ and hence the $h$-vector of $\Delta(\overline{P})$. This is indeed related to interesting conjectures on a family of simplicial spheres called \emph{flag} spheres. A simplicial complex is flag if it coincides with the clique complex of its $1$-skeleton. For example, order complexes of graded posets are flag simplicial complexes. Consider a $(n-1)$-dimensional simplicial complex $\Delta$ which has a \emph{palindromic} $h$-vector (i.e., $h_i(\Delta)=h_{n-i}(\Delta)$ for every $i=0,\dots,n$). For instance, simplicial spheres have this property by Dehn-Sommerville relations. It is easy to see that there exist a unique integer vector $(\gamma_0(\Delta),\dots,\gamma_{\lfloor \frac{n}{2}\rfloor}(\Delta))$ defined by
    \[
        \sum_{i=0}^n h_i(\Delta)t^i = \sum_{i=0}^{\lfloor \frac{n}{2}\rfloor} \gamma_i(\Delta) t^i(t+1)^{n-2i}.
    \]
    For an arbitrary simplicial spheres the numbers $\gamma_i(\Delta)$ do not need to be nonnegative. However, Gal \cite{Gal} conjectured that if $\Delta$ is a flag simplicial sphere (or even a flag homology sphere) then $\gamma_i(\Delta)\geq 0$ for every $i$, which corresponds to a set of linear inequalities in the numbers $h_i(\Delta)$ which generalize a conjecture of Charney and Davis \cite{CD}. Indeed, they conjectured that if $\Delta$ is a flag simplicial $(2k-1)$-sphere, then $(-1)^k\sum_{i=0}^{2k} (-1)^{i}h_i(\Delta)\geq 0$, and this quantity equals $\gamma_{\lfloor \frac{n}{2}\rfloor}(\Delta)$. There is a connection with the $cd$-index of a poset, which was pointed out in \cite[Section 2.3]{Gal}: Indeed we have that
    \begin{equation}\label{eq : Gal cd}
        \sum_{i=0}^{\lfloor \frac{n}{2}\rfloor} \gamma_i(\Delta(\overline{P}))t^i = \Phi_{P}(1,2t),
    \end{equation}
    for any Eulerian poset $P$ of rank $n+1$. In particular, if $n=2k$ (so $P$ is Eulerian), then $\gamma_{k}(\Delta(\overline{P}))=2^k[d^k]\Phi_{P}(c,d)$. Stanley \cite{Stanley94} and Karu \cite{Karu}, respectively, used this idea to show that the Charney-Davis conjecture holds for barycentric subdivisions of simplicial spheres and regular CW-spheres (which are indeed flag), respectively. 
\end{remark}
Combining identity \eqref{eq : Gal cd} with \Cref{cor : manifolds} we obtain the following.
\begin{corollary}
    Let $P$ be an Eulerian Buchsbaum simplicial poset of rank $2k+1$. Then $\gamma_i(\Delta(\overline{P}))\geq 0$, for every $i=0,\dots,k$. In particular,
    \[
     (-1)^k\sum_{i=0}^{2k} (-1)^{i}h_i(\Delta(\overline{P}))\geq 0.
    \]
\end{corollary}

\begin{remark}
    Tracking the coefficient of $d^k$ in the polynomials $P_{j}^{2k}$ in the proof \Cref{thm: PnjPositivity} one can obtain stronger bounds. For instance, for a semi-Eulerian simplicial poset of rank $5$ which is Buchsbaum over a field $\mathbb{F}$ we have
\begin{align*}
    \sum_{i=0}^{4} (-1)^{i}h_i(\Delta(\overline{P}))&\geq 4([d^2]P_{4,2}\widetilde{\beta}_1(\Delta(\overline{P});\mathbb{F})+[d^2]P_{4,3}\widetilde{\beta}_2(\Delta(\overline{P});\mathbb{F}))\\
    &=32\widetilde{\beta}_1(\Delta(\overline{P});\mathbb{F})+16\widetilde{\beta}_2(\Delta(\overline{P});\mathbb{F}).
\end{align*}
    If $\Delta(\overline{P})$ is homeomorphic to a $3$-manifold, then by Poincar\'{e} duality we obtain that $\sum_{i=0}^{4} (-1)^{i}h_i(\Delta(\overline{P})) \geq 48\widetilde{\beta}_1(\Delta(\overline{P});\mathbb{F})$, which should be compared to \cite[Conjecture 5.6]{BOWWZZ} where the authors predict the inequality $\sum_{i=0}^{4} (-1)^{i}h_i(\Delta)\geq 16\widetilde{\beta}_1(\Delta;\mathbb{F})$ for any flag triangulation $\Delta$ of a $3$-manifold.
\end{remark}

\bibliographystyle{plain}
\bibliography{bibliography}

\begin{thebibliography}{10}

\bibitem{Bayer-Signs}
M.~M. Bayer.
\newblock Signs in the {$cd$}-index of {E}ulerian partially ordered sets.
\newblock {\em Proc. Amer. Math. Soc.}, 129(8):2219--2225, 2001.

\bibitem{Bayer-Survey}
M.~M. Bayer.
\newblock The {$cd$}-index: a survey.
\newblock In {\em Polytopes and discrete geometry}, volume 764 of {\em Contemp.
  Math.}, pages 1--19. Amer. Math. Soc., [Providence], RI, [2021] \copyright
  2021.

\bibitem{BaBi}
M.~M. Bayer and L.~J. Billera.
\newblock Generalized {D}ehn-{S}ommerville relations for polytopes, spheres and
  {E}ulerian partially ordered sets.
\newblock {\em Invent Math}, 79:143--157, 1985.

\bibitem{BaKla}
M.~M. Bayer and A.~A. Klapper.
\newblock A new index for polytopes.
\newblock {\em Discrete Comput Geom}, 6:33--47, 1991.

\bibitem{BOWWZZ}
C.~Bibby, A.~Odesky, M.~Wang, S.~Wang, Z.~Zhang, and H.~Zheng.
\newblock Minimal flag triangulations of lower-dimensional manifolds.
\newblock {\em Involve}, 13(4):683--703, 2020.

\bibitem{CD}
R.~Charney and M.~Davis.
\newblock The {E}uler characteristic of a nonpositively curved, piecewise
  {E}uclidean manifold.
\newblock {\em Pacific J. Math.}, 171(1):117--137, 1995.

\bibitem{Ehrenborg-KEUL}
R.~Ehrenborg.
\newblock {$k$}-{E}ulerian posets.
\newblock {\em Order}, 18(3):227--236, 2001.

\bibitem{Ehrenborg-Readdy}
R.~Ehrenborg and M.~Readdy.
\newblock Coproducts and the cd-index.
\newblock {\em Journal of Algebraic Combinatorics}, 8(3):273--299, 1998.

\bibitem{Fine}
J.~Fine.
\newblock Morse theory for convex polytopes.
\newblock {\em manuscript}, 1985.

\bibitem{Foata}
D.~Foata and M.~P. Sch\"utzenberger.
\newblock Nombres d'euler at permutations alternantes.
\newblock {\em In: J.N. Srivastava et al., A Survey of Combinatorial Theory,
  Amsterdam, North-Holland}, pages 173--187, 1973.

\bibitem{Gal}
S.~R. Gal.
\newblock Real root conjecture fails for five- and higher-dimensional spheres.
\newblock {\em Discrete Comput. Geom.}, 34(2):269--284, 2005.

\bibitem{Het}
G.~Hetyei.
\newblock On the {$cd$}-variation polynomials of {A}ndr\'{e} and {S}imsun
  permutations.
\newblock {\em Discrete Comput. Geom.}, 16(3):259--275, 1996.

\bibitem{Karu}
K.~Karu.
\newblock The $cd$-index of fans and posets.
\newblock {\em Compositio Math.}, 142:701--718, 2006.

\bibitem{Klee_1964}
V.~Klee.
\newblock A combinatorial analogue of {P}oincaré’s duality theorem.
\newblock {\em Canadian Journal of Mathematics}, 16:517–531, 1964.

\bibitem{Novik:cd}
I.~Novik.
\newblock Lower bounds for the $cd$-index of odd-dimensional simplicial
  manifolds.
\newblock {\em European Journal of Combinatorics}, 21(4):533--541, 2000.

\bibitem{NS-Socles}
I.~Novik and E.~Swartz.
\newblock Socles of {B}uchsbaum modules, complexes and posets.
\newblock {\em Advances in Mathematics}, 222(6):2059--2084, 2009.

\bibitem{Purtill}
M.~Purtill.
\newblock André permutations, lexicographic shellability and the cd-index of a
  convex polytope.
\newblock {\em Transactions of the American Mathematical Society},
  338(1):77--104, 1993.

\bibitem{Stanley94}
R.~P. Stanley.
\newblock Flag $f$-vectors and the $cd$-index.
\newblock {\em Math. Z.}, 216:483--499, 1994.

\bibitem{Sw:SpheresToManifolds}
E.~Swartz.
\newblock Face enumeration—from spheres to manifolds.
\newblock {\em Journal of the European Mathematical Society}, 011(3):449--485,
  2009.

\end{thebibliography}

\end{document}